\documentclass[a4paper, 11pt]{article}

\usepackage[english]{babel}
\usepackage[T1]{fontenc}
\usepackage{amsmath}
\usepackage{amsthm}
\usepackage{amsfonts}
\usepackage{amssymb}
\usepackage[final]{graphicx}
\usepackage{paralist}
\usepackage{hyperref}
\usepackage{tikz}
\usepackage{booktabs}

\usepackage{cleveref}
\usepackage{subfigure}
\usepackage{pifont}
\usepackage{caption}
\usepackage{mathtools}

\theoremstyle{plain}
\newtheorem{theorem}{Theorem}[section]
\newtheorem{proposition}[theorem]{Proposition}
\newtheorem{lemma}[theorem]{Lemma}
\newtheorem{corollary}[theorem]{Corollary}

\theoremstyle{definition}
\newtheorem{definition}[theorem]{Definition}
\newtheorem{example}[theorem]{Example}
\newtheorem{remark}[theorem]{Remark}

\newcommand{\eps}{\varepsilon}

\DeclareMathOperator{\im}{Im}

\DeclareMathOperator{\re}{Re}

\DeclarePairedDelimiter{\abs}{\lvert}{\rvert}

\DeclarePairedDelimiter{\ceil}{\lceil}{\rceil}
\DeclarePairedDelimiter{\cc}{[}{]} 
\DeclarePairedDelimiter{\co}{[}{[} 

\newcommand{\R}{\ensuremath{\mathbb{R}}}
\newcommand{\C}{\ensuremath{\mathbb{C}}}

\newcommand{\N}{\ensuremath{\mathbb{N}}}

\newcommand{\cC}{\mathcal{C}}
\newcommand{\cM}{\mathcal{M}}

\newcommand{\conj}[1]{\overline{#1}}

\newcommand{\wn}[2]{n(#1; #2)}

\newcommand{\nm}[4]{H_{#3 - #2}^{#4}(#1)}

\DeclareMathOperator{\crit}{crit}

\numberwithin{equation}{section}

\title{\textsc{The transport of images method: computing all zeros of harmonic 
mappings by continuation}}

\author{Olivier S\`{e}te\footnotemark[1] \and Jan Zur\footnotemark[1]}

\begin{document}
\date{May 15, 2021}
\maketitle

\renewcommand{\thefootnote}{\fnsymbol{footnote}}

\footnotetext[1]{TU Berlin, Department of Mathematics, MA 3-3, Stra{\ss}e des 
17.\ Juni 136, 10623 Berlin, Germany.
\texttt{\{sete,zur\}@math.tu-berlin.de}}

\renewcommand{\thefootnote}{\arabic{footnote}}

\begin{abstract}
We present a continuation method to compute all zeros of a harmonic 
mapping $f$ in the complex plane. Our method works without any prior 
knowledge of the number of zeros or their approximate location. 
We start by computing all solution of $f(z) = \eta$ with $\abs{\eta}$ sufficiently large and then 
track all solutions as $\eta$ tends to $0$ to finally obtain all zeros of $f$. 
Using theoretical results on harmonic mappings 
we analyze where and how the number of 
solutions  of $f(z) = \eta$ changes
and incorporate this into the method.  
We prove that our method is guaranteed to compute all zeros, as long as none of them is singular.
In our numerical example the method always terminates with the correct 
number of zeros, is very fast compared to general purpose root finders
and is highly accurate in terms of the residual.
An easy-to-use MATLAB implementation is freely available online.
\end{abstract}

\section{Introduction}\label{sect:introduction} 

We study the zeros of harmonic mappings, i.e., $f:\Omega \to \C$ with $\Delta f 
= 0$ on an open subset $\Omega$ of $\C$.  These functions have a local 
decomposition 
\begin{equation}
f(z) = h(z) + \conj{g(z)}
\end{equation}
with analytic $h,g$, but are themselves non-analytic in general. 
Several results on the number of zeros for special 
classes of harmonic mappings are known in the literature, e.g., for harmonic 
polynomials, i.e., $f(z) = p(z) + \conj{q(z)}$, where $p$ and $q$ are analytic
polynomials~\cite{Wilmshurst1998,KhavinsonSwiatek2003,Geyer2008,
LeeLerarioLundberg2015}, for rational harmonic mappings of the form $f(z) = 
r(z) - \conj{z}$, where $r$ is a rational 
function~\cite{KhavinsonNeumann2006, BleherHommaJiRoeder2014,  
LuceSeteLiesen2014, 
LuceSeteLiesen2015, SeteLuceLiesen2015a, SeteLuceLiesen2015b, LiesenZur2018a, 
LiesenZur2018b, SeteZur2019b}, or certain transcendental harmonic 
mappings~\cite{FassnachtKeetonKhavinson2009,BergweilerEremenko2010,KhavinsonLundberg2010,LuceSete2021}. 

While the above publications have a theoretical focus, we are here interested
in the numerical computation of the zeros of $f$.  A \emph{single} zero can
be computed by an iterative root finder, e.g., Newton's method on $\R^2$
(see~\cite{SeteZur2020} for a convenient formulation for harmonic mappings).
The computation of \emph{all} zeros is much more challenging.  Of course, one
could focus on certain regions in the complex plane and run the root finder
with multiple starting points.  The problem with such an approach is, however,
that in general the number of zeros of a harmonic mapping is not known \emph{a priori}.
In fact, even for harmonic polynomials $f(z)=p(z)-\overline{z}$
with $\deg(p)=n\geq 2$, the number of zeros can vary between $n$ and $3n-2$,
and each of these different numbers of zeros is attained by some harmonic
polynomial; see~\cite[Thm.~5.4]{SeteZur2019b}.

In this paper we present a continuation method to compute all 
zeros of a (non-degenerate) harmonic mapping, provided that the number of zeros 
is finite. We are not aware of any method for this problem in the literature that is 
specialized to harmonic mappings.
Continuation is a general scheme that has been successfully applied
to the numerical solution of systems of nonlinear equations
(see, e.g.,~\cite[Ch.~7.5]{OrtegaRheinboldt1970}
or~\cite{AllgowerGeorg2003}), and in particular for the solution of polynomial systems (see, e.g., \cite{SommeseWampler2005} or~\cite{Morgan2009}).
Moreover, it is used in a wide 
range of applications, including mathematics for economics~\cite[Ch.~5]{Judd1998}, 
chemical physics~\cite{MehtaChenHauensteinWales2014} and astrophysics~\cite[Sect.~10.5]{SchneiderEhlersFalco1999}.
In our continuation approach, the considered function changes only by a 
constant, which is sometimes called \emph{Newton homotopy}.
The overall idea is as follows.
First we solve $f(z) = \eta_1$ for some $\eta_1 \in \C$.
We then construct a sequence $\eta_2, \eta_3, \dots$
and solve $f(z) = \eta_{k+1}$ based on the solutions of $f(z)= \eta_k$.  
We let the sequence $(\eta_k)_k$ tend towards the origin and eventually, after 
finitely many steps, we end with $\eta_n=0$, so that solving $f(z) = \eta_n$ 
yields (numerically) all zeros of~$f$.

Our method, which we call the \emph{transport of images method}, works without any prior knowledge of the number of zeros or their approximate location 
in the complex plane. The main novelty is that we overcome two typical 
difficulties encountered in continuation methods.
\begin{enumerate}
\item \emph{Finding start points:}
To set up the method
all solutions of $f(z) = \eta_1$ for some $\eta_1 \in \C$ are required.
In~\cite[Thm.~3.6]{SeteZur2019b} we proved that for sufficiently large 
$\abs{\eta_1}$ all solutions of $f(z) = \eta_1$ are close to the poles of $f$, 
which are usually known.
Then all solutions can be computed as in~\cite[Sect.~4]{SeteZur2020}.

\item \emph{Determine where and how the number of solutions changes:}
For harmonic mappings, the number of solutions of $f(z) = \eta_k$ depends on 
$\eta_k \in \C$; see, e.g.,~\cite[Remark~5.5]{SeteZur2019b}.
More precisely, the number of solutions of $f(z) = \eta_k$ and $f(z) = 
\eta_{k+1}$ differs by $2$ if $\eta_k$ and $\eta_{k+1}$ are separated by 
a single arc of the caustics of $f$ (curves of the critical values); 
see~\cite{SeteZur2019b}.
We deal with this effect as follows.
For every step from $\eta_k$ to $\eta_{k+1}$ we construct a (minimal) 
set of points $S_{k+1}$, such that all solutions of $f(z) = \eta_{k+1}$ are 
obtained by applying Newton's method with initial points from $S_{k+1}$.
This is based on results of Lyzzaik~\cite{Lyzzaik1992}, 
Neumann~\cite{Neumann2005} and the authors~\cite{SeteZur2019b}.
\end{enumerate} 
We prove that the transport of images method is guaranteed to compute all zeros 
of a non-degenerate harmonic mapping without singular zeros within a finite 
number of continuation steps.

Our numerical examples highlight three key features of our method:
(1)~It terminates with the correct number of zeros.
(2) It is significantly faster than methods that are not problem adapted, such 
as Chebfun2\footnote{Chebfun2, version of September 30, 2020, 
\url{www.chebfun.org}.}.
(3) It is highly accurate in terms of the residual.
An easy-to-use MATLAB implementation\footnote{Transport of images toolbox, 
\url{https://github.com/transportofimages/}.} is freely available online.

The name `transport of images' method is borrowed from astrophysics.
In~\cite[Sect.~10.5]{SchneiderEhlersFalco1999}, the idea is 
roughly described in the context of gravitational lensing.
There, the parameter $\eta$ models the position of a far away light 
source, and its lensed images are represented by the solutions of $f(z)=\eta$ 
with a harmonic mapping $f$;
see the surveys~\cite{KhavinsonNeumann2008, Petters2010, BeneteauHudson2018}. 
However, the description in~\cite{SchneiderEhlersFalco1999} lacks many details 
and contains no further analysis. 

Our paper is organized as follows.  Relevant results on harmonic 
mappings are recalled in Section~\ref{sect:preliminaries}.  In Section~\ref{sect:transport_of_images} 
we study the transport of images method and investigate the solution curves 
of $f(z) = \eta(t)$ depending on~$t$. Key aspects of our implementation are 
discussed in Section~\ref{sect:implementation}.  Section~\ref{sect:numerics} 
contains several numerical examples that illustrate our method.  We close with 
concluding remarks and a brief outlook in Section~\ref{sect:summary}.

\section{Preliminaries}
\label{sect:preliminaries}

We briefly present relevant results for the computation of all zeros of 
harmonic mappings, following the lines of~\cite{SeteZur2019b} 
and~\cite{Lyzzaik1992}.

\subsection{Decomposition of harmonic mappings}

Let $\Omega \subseteq \C$ be open and connected.
A \emph{harmonic mapping} on $\Omega$ is a function $f : \Omega \to \C$ with 
$\Delta f =
4 \partial_{\conj{z}} \partial_z f = 0$, where $\partial_z$ and 
$\partial_{\conj{z}}$ denote the \emph{Wirtinger derivatives}.
If $f$ is harmonic in the 
open disk $D_r(z_0) = \{z \in \C : \abs{z-z_0} < r\}$, it has a local 
decomposition
\begin{equation}\label{eqn:local_decomp}
f(z) = h(z) + \conj{g(z)} = \sum_{k=0}^\infty a_k(z-z_0)^k + 
\conj{\sum_{k=0}^\infty b_k(z-z_0)^k}, \quad z \in D_r(z_0),
\end{equation} 
with analytic functions $h$ and $g$ in $D_r(z_0)$, which are unique up to 
an additive constant; see~\cite[p.~412]{DurenHengartnerLaugesen1996} 
or~\cite[p.~7]{Duren2004}.
If $f$ is harmonic in the punctured disk $D = \{z \in \C: 0 < \abs{z-z_0} < 
r \}$, i.e., $z_0$ is an isolated singularity of $f$, then
\begin{equation}\label{eqn:local_decomp_sing}
f(z) = \sum_{\mathclap{k=-\infty}}^\infty a_k(z-z_0)^k + 
\conj{\sum_{\mathclap{k=-\infty}}^\infty b_k(z-z_0)^k} + 2 A \log \abs{z-z_0},
\quad z \in D;
\end{equation}
see~\cite{SuffridgeThompson2000,HengartnerSchober1987,ArangoArbelaezRivera2020}.
The point $z_0$ is a \emph{pole} of $f$ if $\lim_{z \to z_0} f(z) = \infty$, 
and an \emph{essential singularity} if the limit does not exist in 
$\widehat{\C} = \C \cup \{ \infty \}$; 
see~\cite[p.~44]{Sheil-Small2002}, \cite[Def.~2.1]{SuffridgeThompson2000} 
or~\cite[Def.~2.9]{SeteZur2019b}.
It is possible that $f$ has a pole at $z_0$, but both series 
in~\eqref{eqn:local_decomp_sing} have an essential singularity at $z_0$, as the 
example $f(z) = z^{-2} + \cos(i/z) + \conj{\cos(i/z)}$ 
from~\cite[Ex.~2.3]{SuffridgeThompson2000} shows.  
It is also possible that $f$ has an essential singularity at $z_0$, but both 
series have a pole at $z_0$, as shown by $f(z) = z^{-1} + \conj{z}^{-1}$.

The next proposition is the analog of the well-known facts that a function 
holomorphic on $\widehat{\C}$ is constant and a function meromorphic on 
$\widehat{\C}$ is rational.

\begin{proposition} \label{prop:harmonic_on_sphere}
A harmonic mapping $f$ on $\widehat{\C} \setminus \{ z_1, \ldots, z_m \}$ 
with poles at $z_1, \ldots, z_m \in \widehat{\C}$, at which the principal parts 
of both series in the decomposition~\eqref{eqn:local_decomp_sing} have only 
finitely many terms, has the form
\begin{equation} \label{eqn:ratratlog}
f(z) = r(z) + \conj{s(z)} + \sum_{j=1}^m 2 A_j \log \abs{z-z_j},
\end{equation}
where $A_1, \ldots, A_m \in \C$, and $r$ and $s$ are rational functions that 
can have poles only at $z_1, \ldots, z_m$.
If some $z_j = \infty$, then $\log \abs{z}$ replaces $\log \abs{z-z_j}$.
In particular, if $f$ is harmonic on $\widehat{\C}$, then $f$ is constant.
\end{proposition}

\begin{proof}
First, let $f$ be a harmonic mapping on $\widehat{\C}$.  Then $\partial_z f$ is 
analytic on $\widehat{\C}$
and thus constant~\cite[Thm.~3.5.8]{Wegert2012}.
Let $h' = \partial_z f$, then $h(z) = a_1 z + a_0$ is analytic.
Let $g = \conj{f-h}$.  Since $\partial_{\conj{z}} g = \conj{\partial_z f 
- h'} = 0$, $g$ is analytic on $\widehat{\C}$ and hence constant, say $g(z) = 
b_0$, and we have $f(z) = a_1 z + a_0 + \conj{b}_0$.  Since $f$ is harmonic at 
$\infty$ we have $a_1 = 0$ and $f$ is constant.
Next, let $f$ be harmonic with poles at $z_1, \ldots, z_m$.
If $z_j \neq \infty$ the 
decomposition~\eqref{eqn:local_decomp_sing} has the form
\begin{equation*}
f(z) = \sum_{k=-n}^\infty (a_k (z-z_j)^k + \conj{b_k (z-z_j)^k} ) + 2 A_j \log 
\abs{z-z_j},
\end{equation*}
and we set $f_j(z) = \sum_{k=-n}^{-1} (a_k (z-z_j)^k + \conj{b_k (z-z_j)^k} ) + 
2 A_j \log \abs{z-z_j}$ and similarly if $z_j = \infty$.
Then $f - f_1 - \ldots - f_m$ has removable singularities at $z_1, \ldots, 
z_m$, and can be extended to a harmonic mapping in $\widehat{\C}$ 
(see~\cite[Sect.~2.2]{SeteZur2019b}, or also~\cite[Thm.~15.3d]{Henrici1986}), which 
is constant.
\end{proof}

\subsection{Critical set and caustics of harmonic mappings}
\label{sect:critcaus}

The critical points of a harmonic mapping $f: \Omega \to \C$, i.e., the points 
where the \emph{Jacobian} of $f$,
\begin{equation} \label{eqn:jacobian}
J_f(z) = \abs{\partial_z f(z)}^2 - \abs{\partial_{\conj{z}} f(z)}^2,
\end{equation}
vanishes, form the \emph{critical set} of $f$,
\begin{equation} \label{eqn:crit}
\cC = \{ z \in \Omega : J_f(z) = 0 \}.
\end{equation}
By Lewy's theorem~\cite[p.~20]{Duren2004}
$f$ is locally univalent at $z \in \Omega$ if and only if $z \not\in \cC$.
Here and in the following we assume that $\partial_z f$ and 
$\partial_{\conj{z}} f$ do not vanish identically. (Otherwise $f$ would be 
analytic or anti-analytic, which is not the focus of this paper.)
Then $\partial_z f$ has only isolated zeros in $\Omega$, and the critical set 
consists of isolated points and a level set of the \emph{second complex 
dilatation} of $f$,
\begin{equation}
\omega(z) = \frac{\conj{\partial_{\conj{z}}f(z)}}{\partial_z f(z)},
\end{equation}
which is a meromorphic function in $\Omega$.  We assume that removable 
singularities of $\omega$ in $\Omega$ are removed.
Let
\begin{align}
\cM &= \{ z \in \cC : \abs{\omega(z)} \neq 1 \} \label{eqn:def_cM} \\
&= \{ z \in \Omega : \partial_z f(z) = \partial_{\conj{z}} f (z) = 0 \text{ and 
} 
\lim_{\zeta \to z} \abs{\omega(\zeta)} \neq 1 \},
\end{align}
see \cite[Sect.~2.1]{SeteZur2019b}, which consists of isolated points of 
$\cC$~\cite[Lem.~2.2]{Lyzzaik1992}. Then
\begin{equation}\label{eqn:crit_omega}
\cC \setminus \cM
= \{ z \in \Omega : \abs{\omega(z)} = 1 \}
\end{equation}
is a level set of $\omega$.
If $J_f \not\equiv 0$, then $\cC \setminus \cM$ consists of analytic curves.
These intersect in $z \in \cC \setminus \cM$ if, and only if, $\omega'(z) = 0$.
At $z \in \cC \setminus \cM$ with $\omega'(z) \neq 0$, the equation
\begin{equation} \label{eqn:parametrization}
\omega(\gamma(t)) = e^{it}, \quad t \in I \subseteq \R,
\end{equation}
implicitly defines a local analytic parametrization $z = \gamma(t)$ of $\cC 
\setminus \cM$.

\begin{definition}
We call the set of critical values of $f$, i.e., $f(\cC)$,
the \emph{caustics} or the \emph{set of caustic points}.
The caustics induce a partition of $\C \setminus f(\cC)$, and
we call a connected component $A \subseteq \C \setminus f(\cC)$ with 
$\partial A \subseteq f(\cC)$ a \emph{caustic tile}.
\end{definition}

The next lemma characterizes a tangent vector to the caustics; see 
\cite[Lem.~2.1]{SeteZur2019b} or \cite[Lem.~2.3]{Lyzzaik1992}.

\begin{lemma} \label{lem:tangent_to_caustic}
Let $f$ be a harmonic mapping, $z_0 \in \cC \setminus \cM$ with $\omega'(z_0) 
\neq 0$ and let $z_0 = \gamma(t_0)$ with the 
parametrization~\eqref{eqn:parametrization}.
Then $f \circ \gamma$ is a parametrization of a caustic and the
corresponding tangent vector at $f(z_0)$ is
\begin{equation} \label{eqn:tangent_to_caus}
\tau(t_0) = \frac{d}{dt} (f \circ \gamma)(t_0)
= e^{- i t_0/2} \psi(t_0)
\end{equation}
with 
\begin{equation} \label{eqn:psi}
\psi(t) = 2 \re( e^{i t/2} \, \partial_z f(\gamma(t)) \, \gamma'(t)).
\end{equation}
In particular, the rate of change of the argument of the tangent vector is
\begin{equation}
\frac{d}{dt} \arg(\tau(t)) \big\vert_{t = t_0} = - \frac{1}{2}
\end{equation}
at points where $\psi(t_0) \neq 0$,
i.e., the curvature of the caustics is constant with respect to the 
parametrization $f \circ \gamma$.
Moreover, $\psi$ has either only finitely many zeros, or is identically zero, 
in which case $f \circ \gamma$ is constant.
\end{lemma}

\begin{definition}[{\cite[Def.~2.2]{SeteZur2019b}}]
In the notation of Lemma~\ref{lem:tangent_to_caustic} we call $f(\gamma(t_0))$ 
a \emph{fold caustic point} or simply a \emph{fold}, if the tangent $\tau(t_0)$ 
is non-zero, i.e., $\psi(t_0) \neq 0$, and a \emph{cusp} if $\psi$ has a zero with a sign change at~$t_0$.
\end{definition}

This classification of caustic points is not complete but sufficient for our 
needs.
Fold points form open arcs of the caustic curves, since $\psi$ is continuous.
At a cusp the argument of the tangent vector changes by $+\pi$.
The cusps, the curvature of the caustics and the four caustic 
tiles are apparent in 
Figure~\ref{fig:counting_example} (left).

\subsection{Pre-images under non-degenerate harmonic mappings}

The position of $\eta \in \C$ with respect to the caustics affects the number 
of pre-images of $\eta$ under $f$.
The following wide class of harmonic mappings was introduced 
in~\cite[Def.~3.1]{SeteZur2019b}.

\begin{definition} \label{def:nondegenerate}
We call a harmonic mapping $f$ \emph{non-degenerate} on $\widehat{\C}$ if
\begin{enumerate}
\item $f$ is defined in $\widehat{\C}$ except at finitely many poles $z_1, 
\dots, z_m \in \widehat{\C}$,
\item at each pole $z_j \in \C$ of $f$, the 
decomposition~\eqref{eqn:local_decomp_sing} of $f$ has the form
\begin{equation} \label{eqn:near_pole}
f(z) = \sum_{k=-n}^\infty a_k(z-z_j)^k + \conj{\sum_{k=-n}^\infty b_k(z-z_j)^k} 
+ 2 A_j \log\abs{z-z_j}
\end{equation}
with $n\ge 1$ and $\abs{a_{-n}} \neq \abs{b_{-n}}$,
and if $z_j = \infty$ is a pole of $f$, then
\begin{equation} \label{eqn:near_pole_infty}
f(z) = \sum_{k=-\infty}^n a_kz^k + \conj{\sum_{k=-\infty}^n b_kz^k} + 2 A_j 
\log\abs{z}, \quad \text{for } \abs{z} > R,
\end{equation}
with $n\ge 1$ and $\abs{a_{n}} \neq \abs{b_{n}}$, and $R > 0$,
\item the critical set of $f$ is bounded in $\C$.
\end{enumerate}
\end{definition}

The intensively studied (see references in the introduction) rational harmonic 
mappings $r(z) - \conj{z}$ and harmonic polynomials 
$p(z) + \conj{q(z)}$ are non-degenerate if, and only if,
$\lim_{z \to \infty} \abs{r(z)/z} \neq 1$ and 
$\lim_{z \to \infty} \abs{p(z)/q(z)} \neq 1$, respectively, since
$\infty$ is the only pole of $p(z)$, $q(z)$ and $z$.

A non-degenerate harmonic mapping $f$ has a global 
decomposition~\eqref{eqn:ratratlog}, and
\begin{equation} \label{eqn:dzdzbarf}
\partial_z f(z) = r'(z) + \sum_{j=1}^m \frac{A_j}{z - z_j}, \quad
\conj{\partial_{\conj{z}} f(z)} = s'(z) + \sum_{j=1}^m \frac{\conj{A}_j}{z - 
z_j},
\end{equation}
are rational functions.
By item 2., the poles are not accumulation points of the critical 
set~\cite[Lem.~2.2, 2.3]{SuffridgeThompson2000}.
Hence, $\cC$ and $f(\cC)$ are compact sets in $\C$.
Item 3.\ implies that $J_f \not\equiv 0$.
If $\partial_z f \equiv 0$, then $\conj{\partial_{\conj{z}} f} \not\equiv 0$ 
(since $J_f \not\equiv 0$), and $\cC$ consists of the finitely many zeros 
of the rational function $\conj{\partial_{\conj{z}} f}$.
If $\partial_z f \not\equiv 0$, then $\omega = \conj{\partial_{\conj{z}} f} / 
\partial_z f$ is rational, $\cM$ has only finitely many points, and $\cC 
\setminus \cM$ consists of the $\deg(\omega)$ many pre-images of the unit 
circle under $\omega$.
We can parametrize $\cC \setminus \cM$ with analytic closed curves according 
to~\eqref{eqn:parametrization}, such that every $z \in \cC\setminus\cM$ with 
$\omega'(z) \neq 0$ belongs to exactly one curve;
see~\cite[Seect.~3.1]{SeteZur2019b}.  We call these curves the \emph{critical 
curves} of $f$ and denote the set of all of them by $\crit$.
The critical curves and their images under $f$ both have finite total length 
since their parametrization is piecewise $C^\infty$.

Let $P(f) = n_1 + \ldots + n_m$ be the sum of the orders of the poles $z_1, 
\ldots, z_m$ in~\eqref{eqn:near_pole} or~\eqref{eqn:near_pole_infty}.
Then $N_\eta(f) = \abs{ \{ z \in \widehat{\C} : f(z) = \eta \} }$, 
the number of pre-images of $\eta$ under $f$, relates to 
$P(f)$ and the winding number $n(f \circ \gamma; \eta)$ of the caustics about 
$\eta$ as follows.

\begin{theorem}[{\cite[Thms.~3.4, 3.7]{SeteZur2019b}}]
\label{thm:relative_counting}
Let $f$ be a non-degenerate harmonic mapping on $\widehat{\C}$.
Then for a point $\eta \in \C \setminus f(\cC)$ we have
\begin{equation}
N_\eta(f) = P(f) + 2 \sum_{\gamma \in \crit} n(f \circ \gamma; \eta).
\end{equation}
In particular, the number of pre-images of $\eta \in \C \setminus f(\cC)$ under 
$f$ is finite.
Moreover, if $\eta_1, \eta_2 \in \C \setminus f(\cC)$ 
then
\begin{equation}
N_{\eta_2}(f) = N_{\eta_1}(f) + 2 \sum_{\gamma \in \crit} \big( \wn{f \circ 
\gamma}{\eta_2} - \wn{f \circ \gamma}{\eta_1} \big),
\end{equation}
and we have:
\begin{enumerate}
\item If $\eta_1$ and $\eta_2$ are in the same caustic tile then the number 
of pre-images under $f$ is the same, i.e., $N_{\eta_2}(f) = N_{\eta_1}(f)$.

\item If $\eta_1$ and $\eta_2$ are separated by a single caustic arc $f \circ 
\gamma$ then the number of pre-images under $f$ differs by $2$, i.e., 
$N_{\eta_2}(f) = N_{\eta_1}(f) \pm 2$.
\end{enumerate} 
\end{theorem}

\begin{example}[{\cite[Ex.~3.10]{SeteZur2019b}}] \label{ex:log}
We study the non-degenerate harmonic mapping
\begin{equation}\label{eqn:example_log}
f(z)
= z^2 + \conj{z^{-1} + (z+1)^{-1}} + 2\log\abs{z};
\end{equation}
see Figure~\ref{fig:counting_example}.
By Lemma~\ref{lem:tangent_to_caustic} the curvature of a caustic is always 
negative.
Consequently, the winding numbers of the 
only caustic $f\circ\gamma$ about $\eta_1, \eta_2$ and $\eta_3$ are 
$0$, $-1$ and $1$ respectively.  By Theorem~\ref{thm:relative_counting} we have 
\begin{equation*}
\begin{split}
N_{\eta_1}(f) &= P(f) + 2 n(f \circ \gamma; \eta_1) = 4 + 0 = 4, \\
N_{\eta_2}(f) &= N_{\eta_1}(f) + 2(n(f\circ\gamma;\eta_2) - 
n(f\circ\gamma;\eta_1)) = 4 - 2 = 2, \\
N_{\eta_3}(f) &= N_{\eta_1}(f) + 2(n(f\circ\gamma;\eta_3) - 
n(f\circ\gamma;\eta_1)) = 4 + 2 = 6, 
\end{split}
\end{equation*}
as we see in Figure~\ref{fig:counting_example}.
\end{example}

\begin{figure}[t]
{\centering
\includegraphics[width=0.24\linewidth,height=0.24\linewidth]
{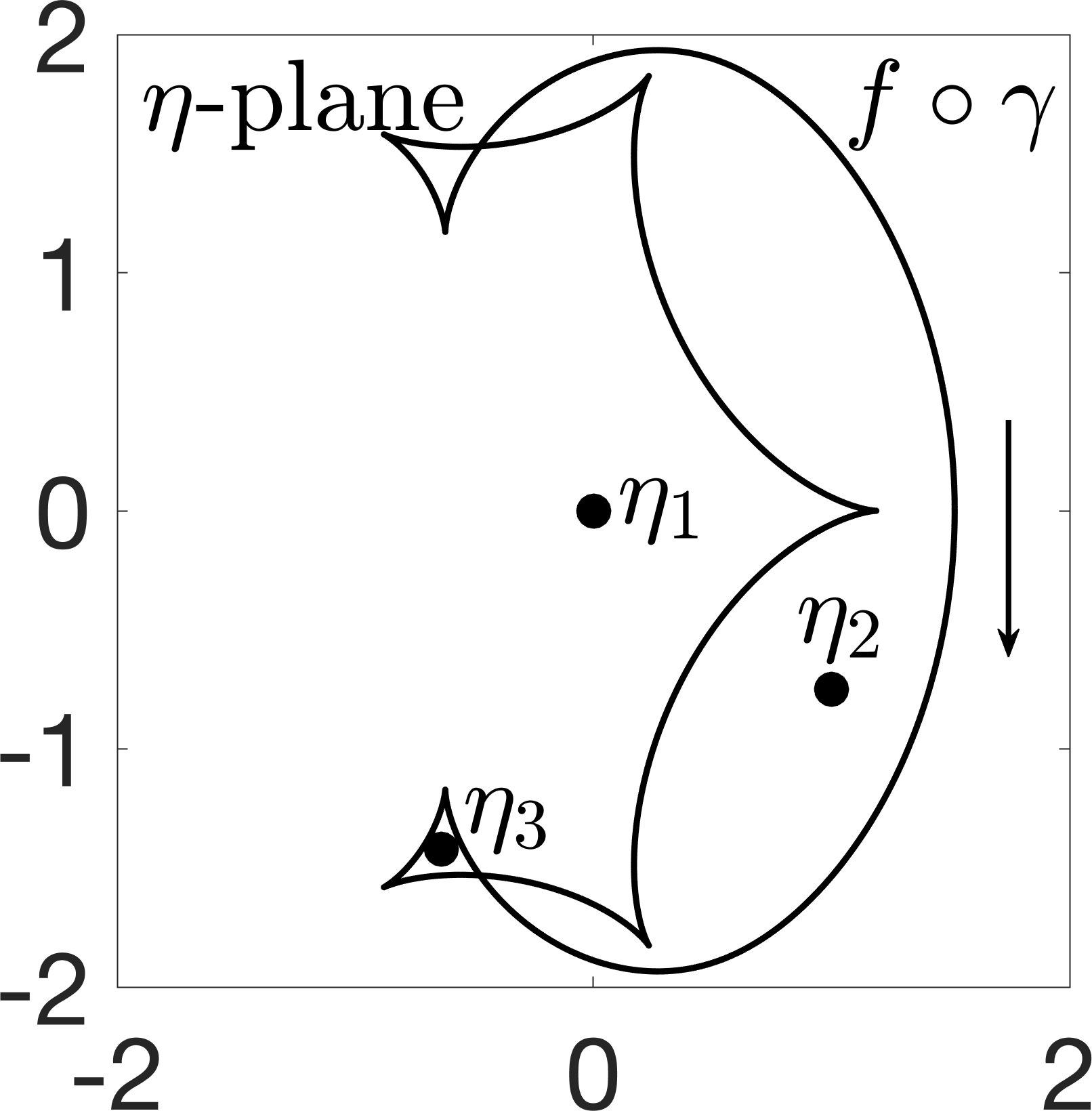}
\includegraphics[width=0.24\linewidth,height=0.24\linewidth]
{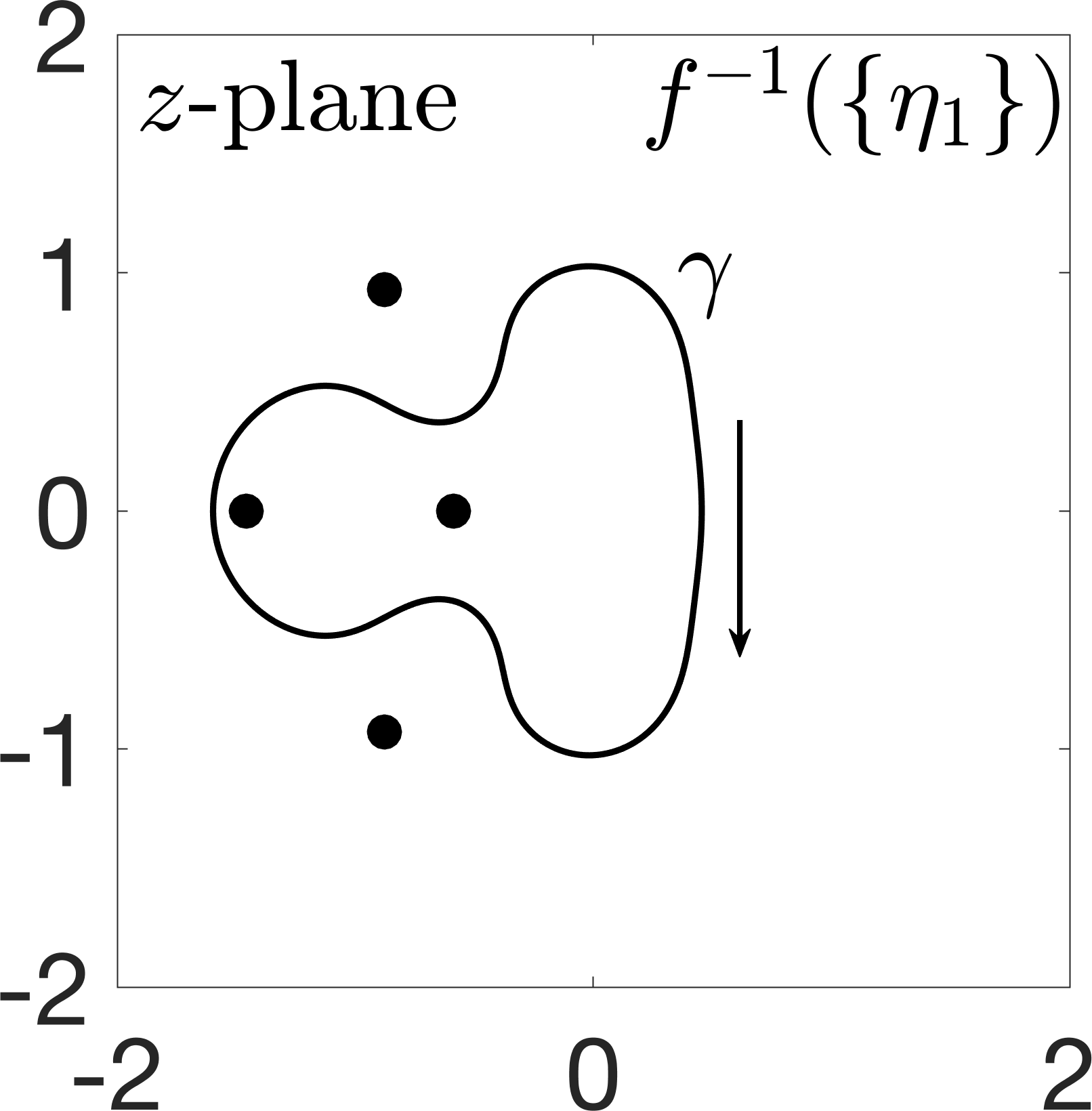}
\includegraphics[width=0.24\linewidth,height=0.24\linewidth]
{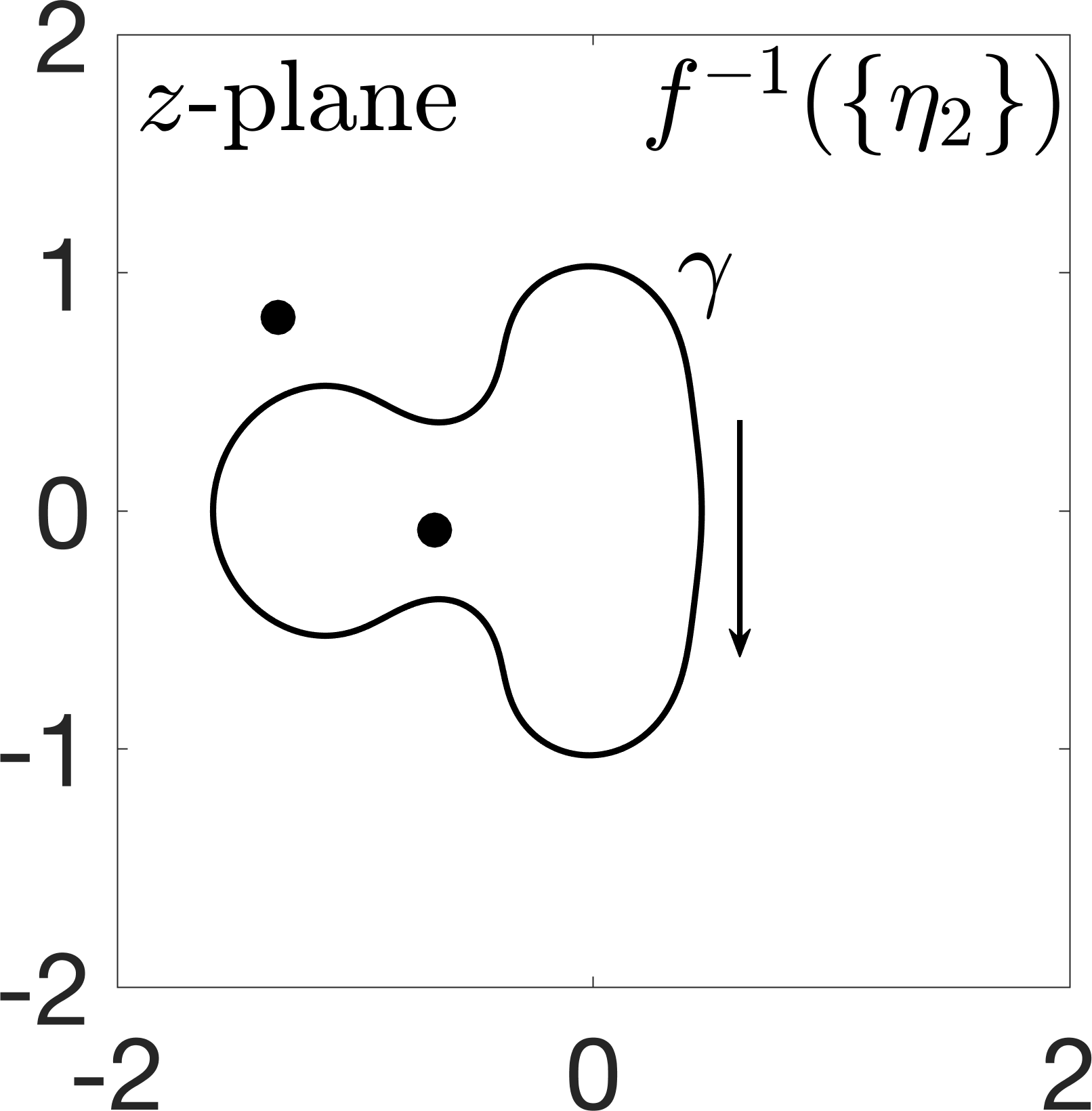}
\includegraphics[width=0.24\linewidth,height=0.24\linewidth]
{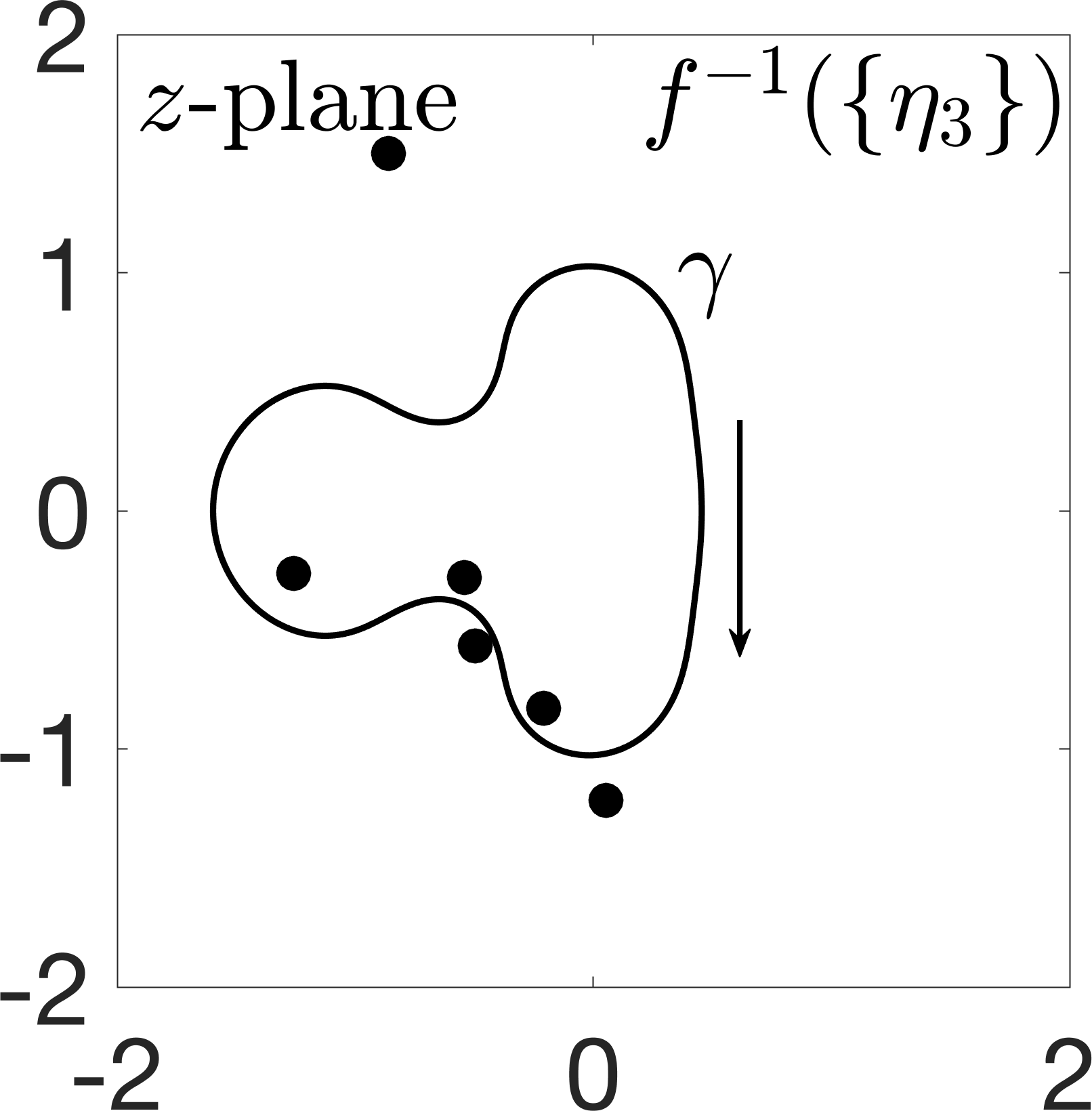}

}
\caption{Harmonic mapping~\eqref{eqn:example_log}. 
$\eta$-plane: the caustic $f \circ \gamma$ of $f$ and the points $\eta_j$.
$z$-plane: pre-images of $\eta_j$ under $f$. 
The arrows indicate the orientation 
of the critical curve $\gamma$ and the caustic $f \circ \gamma$.}
\label{fig:counting_example}
\end{figure}

\section{The transport of images method}\label{sect:transport_of_images}

We compute all zeros of a non-degenerate harmonic mapping $f$ by continuation, 
provided that no zero of $f$ is singular, or equivalently, if $0$ belongs to the caustics of $f$.
A zero $z_0$ is \emph{singular}, if $J_f(z_0) = 0$ holds.
Recall from the introduction that
we successively compute all solutions of $f(z) = \eta_k$ for a sequence 
$\eta_1, \eta_2, \ldots$ that tends to zero.
During this procedure the number of solutions may change, 
depending on the positions of the points $\eta_k$ relative to the caustics of 
$f$; see Theorem~\ref{thm:relative_counting}.
We start by computing all solutions of $f(z) = \eta_1$ for  
sufficiently large $\abs{\eta_1}$ from the local decomposition \eqref{eqn:local_decomp_sing} 
of $f$ near the poles.
Given all solutions of $f(z) = \eta_k$ we determine an $\eta_{k+1} \in \C$ and a set
$S_{k+1} \subset \C$ (prediction) such that the following holds:
(1) for every solution $z_* \in \C$ of $f(z) = \eta_{k+1}$ there exists a $z_0 
\in S_{k+1}$ such that the sequence $(z_j)_{j \in \N}$ of Newton iterates
for $f - \eta_{k+1}$ converges to~$z_*$ (correction);
(2) the number of solutions of $f(z) = \eta_{k+1}$ coincides with the number of 
elements in $S_{k+1}$, i.e., $S_{k+1}$ is minimal.
When this prediction-correction scheme reaches $\eta_n = 0$ we obtain 
\emph{all} zeros of~$f$.
This leads to the following algorithm.

\paragraph{The transport of images method}
\begin{compactenum}
\item \emph{Initial phase:} compute all solutions of $f(z) = \eta_1$.

\item \emph{Transport phase:} while $\eta_k \neq 0$ do
\begin{compactenum}
\item \emph{Prediction:} choose $\eta_{k+1} \in \C$ and construct a minimal set 
of 
initial points $S_{k+1}$ from the solutions of $f(z) = \eta_k$ and from the 
caustics.

\item \emph{Correction:} apply Newton's method to $f- \eta_{k+1}$ 
with the set of initial points $S_{k+1}$ to get all solutions of $f(z) = 
\eta_{k+1}$. 
\end{compactenum}
\end{compactenum}
\bigskip

The points $\eta_1, \eta_2, \ldots, \eta_n = 0$ conceptually form a path in 
$\C$ (see Figure~\ref{fig:transportpath}).
From this perspective
the transport of images method discretizes the tracing of the solution curves 
of $f(z) = \eta(t)$, where the right hand side is a path $\eta : \cc{a, b} \to 
\C$.
We describe the homotopy curves for a general harmonic mapping in 
Section~\ref{sect:homotopycurves} in full detail.

\begin{figure}
\begin{center}
\begin{tikzpicture}
\usetikzlibrary{positioning,arrows}
\tikzset{arrow/.style={-latex}}

\draw[fill] (1,2) circle (2pt) node[above] {};
\draw[fill] (0,1) circle (2pt) node[below] {$\eta_1$};
\draw[fill] (2,1.5) circle (2pt) node[above] {$\eta_k$};
\draw[fill] (3.3,1.5) circle (2pt) node[above] {$\eta_{k+1}$};
\draw[fill] (4,1.5) circle (2pt) node[above] {};
\draw[fill] (5,1) circle (2pt) node[above] {$\eta_\ell$};
\draw[fill] (5.7,0.3) circle (2pt) node[left] {$\eta_{\ell+1}$};
\draw[fill] (6,0) circle (2pt) node[above] {};
\draw[fill] (6.5,0.25) circle (2pt) node[above] {$\eta_{m}$};
\draw[fill] (7.5,0.75) circle (2pt) node[above] {$\eta_{m+1}$};
\draw[fill] (8,1) circle (2pt) node[above] {};
\draw[fill] (9.2,0) circle (2pt) node[above] {};
\draw[fill] (10,0) circle (2pt) node[above] {$\eta_n$};

\draw (4,0)node[below] {caustics};
\draw[densely dotted] plot coordinates {(3.3,-0.1) (2.5,0.5)};
\draw[densely dotted] plot coordinates {(4.8,-.3) (6.8,-.5)};

\draw [thick] (3.4,2.5) arc (130:180:3);

\draw [thick] (6.9,-0.7) arc (-20:30:3);

\draw[dashed] plot coordinates {(0,1) (1,2) (2,1.5) (4,1.5) (5,1) (6,0) (8,1) 
(9.2,0) (10,0)};

\draw [arrow] (3.3,1.3) to  (2,1.3) node[below] {$c$};
\draw [arrow] (6.6,0.05) to  (7.6,0.55) node[below] {$c$};

\end{tikzpicture}
\end{center}
\caption{A path from $\eta_1$ to $\eta_n$ with the three kinds of steps in the 
transport of images method:
(1) $f(z) = \eta_{k+1}$ has two solutions fewer than $f(z) = \eta_k$,
(2) the number of solutions for $\eta_\ell$ and $\eta_{\ell+1}$ coincides
and (3) $f(z) = \eta_{m+1}$ has two solutions more than $f(z) = \eta_m$.
}
\label{fig:transportpath}
\end{figure}
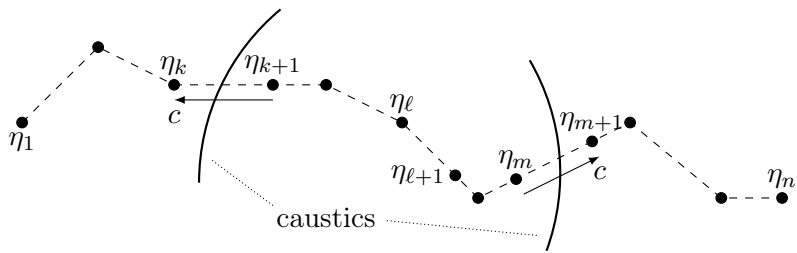

\subsection{Newton's method as corrector}\label{sect:harmonic_newton}

For a harmonic mapping $f: \Omega \to \C$ the \emph{harmonic Newton map} is
\begin{equation} \label{eqn:newton_map}
H_f : \Omega \setminus \cC \to \C, \quad
H_f(z) = z - \frac{\conj{h'(z)} f(z) - \conj{g'(z) 
f(z)}}{\abs{h'(z)}^2 - \abs{g'(z)}^2},
\end{equation}
where $f = h + \conj{g}$ is a local decomposition~\eqref{eqn:local_decomp}
and the \emph{harmonic Newton iteration} with initial point $z_0 \in \Omega 
\setminus \cC$ is
\begin{equation} \label{eqn:harmonic_newton}
z_k = H_f(z_{k-1}) = H_f^k(z_0), \quad k \geq 1.
\end{equation}
We write, as usual, $H_f^k(z) = H_f(H_f^{k-1}(z))$ for $k \geq 1$ and $H_f^0(z) 
= z$.
The iteration~\eqref{eqn:harmonic_newton} is a complex formulation of the 
classical Newton iteration 
in~$\R^2$ \cite[Sect.~3.1]{SeteZur2020} and hence inherits all its properties, 
e.g., local quadratic convergence.

To compute solutions of $f(z) = \eta$ with $\eta \in \C$ we apply the 
iteration~\eqref{eqn:harmonic_newton} to $f - \eta$.
The limit of the harmonic Newton iteration with initial point $z_0$, 
\begin{equation}\label{eqn:correction}
\nm{z_0}{\eta}{f}{\infty} = \lim_{k \to \infty} \nm{z_0}{\eta}{f}{k},
\end{equation}
exists for all $z_0$ in a
neighborhood of a solution $z_* \notin \cC$ of $f(z) = \eta$ by the local 
convergence property of Newton's method.
If $z_* = \nm{z_0}{\eta}{f}{\infty}$ holds we say that $z_*$ \emph{attracts} 
$z_0$ under $H_{f-\eta}$ or that $z_0$ \emph{is attracted} by $z_*$ under 
$H_{f-\eta}$.

\begin{definition}
For a harmonic mapping $f$ we call a set $S$ 
a \emph{(minimal) prediction set} of $f^{-1}(\{ \eta \})$ if every 
solution of $f(z) = \eta$ attracts exactly one point in $S$ under $H_{f-\eta}$, 
i.e., if $\nm{S}{\eta}{f}{\infty} = f^{-1}(\{ \eta \})$,
and $\abs{S} = \abs{f^{-1}(\{ \eta \})}$,
where $\abs{X}$ denotes the cardinality of a set $X$.
\end{definition}

For a prediction set $S$ the map $H_{f - \eta}^\infty$ is a bijection
from $S$ to $f^{-1}( \{ \eta \})$.
In this sense every point $z_0 \in S$ represents exactly one solution of $f(z) 
= \eta$.
In~\cite{Smale1986} the point $z_0$ is called an \emph{approximate zero} of
$f - \eta$ if the Newton iterates additionally satisfy
$\abs{z_k - z_{k-1}} \leq (\frac{1}{2})^{2^{k-1}-1} \abs{z_1 - z_0}$.

\subsection{Initial phase}
\label{sect:initial_phase}

For non-degenerate harmonic mappings the initial phase can be realized as 
follows.  For sufficiently large $\abs{\eta}$ all solutions of $f(z) = \eta$ 
are close to the poles of $f$ and, in the notation~\eqref{eqn:near_pole} 
or~\eqref{eqn:near_pole_infty}, there are exactly $n$ solutions close to the 
pole $z_j$~\cite[Thm.~3.6]{SeteZur2019b}.  We compute these solutions with 
the harmonic Newton iteration and the initial points from the next theorem.
A similar result holds if $\infty$ is a pole of $f$, 
generalizing~\cite[Cor.~4.5]{SeteZur2020}.
All initial points together give a prediction set of $f^{-1}(\{\eta\})$.

\begin{theorem} \label{thm:zeros_at_poles}
Let $f$ be harmonic in $D = \{ z \in \C : 0 < \abs{z - z_0} < r \}$ with
\begin{equation*}
f(z) = \sum_{k=-n}^\infty a_k (z-z_0)^k + 
\conj{\sum_{k=-n}^\infty b_k (z-z_0)^k} + 2 A \log \abs{z - z_0}, \quad z \in D,
\end{equation*}
where $n \geq 1$ and $\abs{a_{-n}} \neq \abs{b_{-n}}$.
Suppose that $c = \eta - (a_0 + \conj{b}_0) \neq 0$ and
let $\zeta_1, \ldots, \zeta_n$ be the $n$ solutions of
\begin{equation} \label{eqn:init_at_pole}
(\zeta - z_0)^n = \frac{\abs{a_{-n}}^2 - \abs{b_{-n}}^2}{\conj{a}_{-n} c - 
\conj{b}_{-n} \conj{c}}.
\end{equation}
We then have for sufficiently large $\abs{c}$, i.e., for sufficiently large 
$\abs{\eta}$:
\begin{enumerate}
\item There exist exactly $n$ solutions of $f(z) = \eta$ near $z_0$.
\item The set $S = \{ \zeta_1, \ldots, \zeta_n \}$ is a prediction set for the 
solutions in 1.
\end{enumerate}
\end{theorem}

\begin{proof}
If $A = 0$ then $f - \eta$ has $n$ distinct zeros near $z_0$, which satisfy 
2.\ by~\cite[Thm.~4.3]{SeteZur2020}.  For $A \neq 0$ the proof closely 
follows the proof of~\cite[Thm.~4.3]{SeteZur2020} if one replaces $h'$ by 
$\partial_z f$, and $g'$ by $\conj{\partial_{\conj{z}} f}$.
As in the proof of~\cite[Thm.~3.6]{SeteZur2019b},
Rouch\'e's theorem (e.g.~\cite[Thm.~2.3]{SeteLuceLiesen2015a}) implies that
$f - \eta$ cannot have more than $n$ zeros in a neighborhood of $z_0$.
\end{proof}

\subsection{Transport phase}

We analyze the prediction-correction scheme of the transport phase and prove 
that the transport phase can be realized with a finite number of steps.
We begin with a single step and
investigate how all solutions of $f(z) = \eta_2$ can be determined, given 
all solutions of $f(z) = \eta_1$.
First we prove that a solution of $f(z) = \eta_1$ that is not in $\cC$
is attracted by a solution of $f(z) = \eta_2$ under the harmonic Newton map 
$H_{f-\eta_2}$ if $\abs{\eta_2-\eta_1}$ is small enough.

\begin{lemma} \label{lem:transport1}
Let $f : \Omega \to \C$ be a harmonic mapping, $z_0 \in \Omega \setminus \cC$ 
and $\eta_1 = f(z_0)$.
Then there exist $\eps > 0$ and $\delta > 0$ such that for each $\eta_2 \in 
D_\eps(\eta_1)$ there exists a unique $z_* \in D_\delta(z_0)$ with $f(z_*) = 
\eta_2$.
Moreover, $z_*$ attracts $z_0$ under the harmonic Newton map 
$H_{f-\eta_2}$, i.e., $\nm{z_0}{\eta_2}{f}{\infty} = z_*$.
\end{lemma}

\begin{proof}
Let $\zeta \in \Omega \setminus \cC$ and consider $f - f(\zeta)$.  The 
Newton--Kantorovich theorem yields a radius $\rho(\zeta) > 0$ so that $\zeta$ 
attracts all points in $D_{\rho(\zeta)}(\zeta)$ under the harmonic Newton map 
of $f - f(\zeta)$; see~\cite[Thm.~2.2, Sect.~4]{SeteZur2020} and references 
therein.  (The basin of attraction of $\zeta$ might be larger.)
Note that $\rho$ depends continuously on $\zeta$.

There exist open neighborhoods $U \subseteq \Omega \setminus \cC$ of $z_0$ and 
$V$ of $\eta_1$ such that $f : U \to V$ is a diffeomorphism (inverse function 
theorem).
Let $r > 0$ with $K = \overline{D_r(z_0)} \subseteq U$ and let $m = \min_{z 
\in K} \rho(z)$.
Then $m > 0$ since $K$ is compact and $\rho$ is continuous with $\rho(z) > 0$ 
for $z \in U$.
Let $0 < \delta \leq \min \{ r, m \}$.  Then $z_* \in D_\delta(z_0)$ satisfies 
$\abs{z_* - z_0} < \delta \leq m \leq \rho(z_*)$, hence 
$\nm{z_0}{f(z_*)}{f}{\infty} = z_*$.
Finally, there exists $\eps > 0$ with $D_\eps(\eta_1) \subseteq 
f(D_\delta(z_0))$.
\end{proof}

By Theorem~\ref{thm:relative_counting}, two points in the same caustic tile have 
the same number of pre-images.  Moreover, $\eta \in \C \setminus
f(\cC)$ has only finitely many pre-images, and these are in $\widehat{\C} 
\setminus \cC$.
Thus, we can apply Lemma~\ref{lem:transport1} to all solutions simultaneously 
to obtain the next theorem.

\begin{theorem} \label{thm:correction_regular}
Let $f$ be a non-degenerate harmonic mapping on $\widehat{\C}$ and let 
$\eta_1 \in \C \setminus f(\cC)$.  Then there exists an $\eps > 0$ such that 
$S = f^{-1}(\{ \eta_1 \})$ is a prediction set of $f^{-1}(\{ \eta_2 \})$ for 
all $\eta_2 \in D_{\eps}(\eta_1) \subseteq \C \setminus f(\cC)$.
\end{theorem}

By Theorem~\ref{thm:relative_counting}, the number of pre-images of $\eta_1, 
\eta_2$ differs by $2$ if $\eta_1, \eta_2$ are in adjacent caustic tiles 
separated by a single caustic arc.
The two additional solutions can be computed with the next theorem, provided 
that the 
step from $\eta_1$ to $\eta_2$ crosses the caustics in a specific direction.

\begin{theorem}[{\cite[Thm.~4.2]{SeteZur2019b}}] \label{thm:zero_at_crit}
Let $f : \Omega \to \C$ be a harmonic mapping and $z_0 \in \cC \setminus 
\cM$, such that 
$\eta = f(z_0)$ is a fold caustic point.  Moreover, let
\begin{equation} \label{eqn:c}
f(z) = \sum_{k=0}^\infty a_k (z-z_0)^k + \conj{\sum_{k=0}^\infty b_k (z-z_0)^k}
\quad \text{and }
c = - \left( \frac{a_2 \conj{b}_1}{a_1} + \frac{\conj{b}_2 a_1}{\conj{b}_1} 
\right).
\end{equation}
Then, for each sufficiently small $\eps > 0$, there exists a $\delta > 0$ such 
that for all $0 < t < \delta$:
\begin{enumerate}
\item $f(z) = \eta - t c$ has no solution in $D_\eps(z_0)$,
\item $f(z) = \eta$ has exactly one solution in $D_\eps(z_0)$,
\item $f(z) = \eta + t c$ has exactly two solutions in $D_\eps(z_0)$.
\end{enumerate}
Moreover, each solution in 3.\ attracts one of the points
\begin{equation}
z_\pm = z_0 \pm i \sqrt{t \conj{b}_1/a_1}\label{eqn:zpm}
\end{equation}
under the harmonic Newton map $H_{f - (\eta + tc)}$
if $t > 0$ is sufficiently small.
\end{theorem}

Note that $f$ is \emph{light} (i.e., $f^{-1}(\{ \eta \})$ is either totally 
disconnected or empty for all $\eta \in \C$) in a neighborhood of $z_0$ if 
$f(z_0)$ is a 
fold~\cite[Rem.~2.3]{SeteZur2019b}; hence we omit this assumption 
from~\cite[Thm.~4.2]{SeteZur2020}.
We emphasize that $c$ is not necessarily a normal vector to the caustics.

With Theorem~\ref{thm:zero_at_crit}, we construct a prediction 
set of $f ^{-1}(\{ \eta_2 \})$ if the step from $\eta_1$ to $\eta_2$ 
crosses a caustic at a \emph{simple fold} $\eta$, i.e., $\abs{f^{-1}(\{ \eta \}) 
\cap \cC} = 1$.

\begin{theorem}\label{thm:correction_singular}
Let $f$ be a non-degenerate harmonic mapping on $\widehat{\C}$.
Let $\eta_1, \eta_2 \in \C \setminus f(\cC)$ such that there exists exactly 
one simple fold $\eta$ and no other caustic point on the line 
segment from $\eta_1$ to $\eta_2$, and such that $\arg(\eta_{2} - \eta_1) = 
\arg(\pm c)$, where $c$ is defined as in Theorem~\ref{thm:zero_at_crit}.
Then for small enough $\abs{\eta_2 - \eta_1}$,
\begin{equation}\label{eqn:Sk_cross}
S = \begin{cases}f^{-1}(\{\eta_1\}) \cup\{z_\pm \}, & \text{ if }
\arg(\eta_{2} - \eta_1) = \arg(+c), \\ 
f^{-1}(\{\eta_1\})\setminus\nm{ \{z_\pm \} }{\eta_1}{f}{\infty}, & \text{ 
if } \arg(\eta_{2} - \eta_1) = \arg(-c) ,\end{cases}
\end{equation}
is a prediction set of $f^{-1}(\{ \eta_2 \})$,
with $z_\pm$ from~\eqref{eqn:zpm}.
\end{theorem}

\begin{proof}
Let $\eta = f(z_0)$ be the unique simple fold on the line segment 
between $\eta_1$ and $\eta_2$.
The step from $\eta_1$ to $\eta_2$ produces either two additional or two fewer 
solutions in a neighborhood of $z_0$, depending on $c$; see 
Theorem~\ref{thm:zero_at_crit}.  The global number of solutions changes 
accordingly; see Theorem~\ref{thm:relative_counting}.

In the case of two additional solutions, i.e., if $\arg(\eta_2 - \eta_1) = 
\arg(+c)$, the set $S$ contains the solutions of $f(z) = \eta_1$, 
and the points $z_\pm$, hence $\abs{S} = \abs{f^{-1}(\{ \eta_2 \})}$ for 
sufficiently small $\abs{\eta_2 - \eta_1}$.
Let $t_1 , t_2 > 0$ such that $\eta_1 = \eta - t_1 c$ and $\eta_2 = \eta + t_2 
c$.  In particular, we have $t_j \le \abs{(\eta_2 - \eta_1)/c}$.
In the disk $D_\eps(z_0)$ from Theorem~\ref{thm:zero_at_crit},
$f(z) = \eta_1$ has no solution and $f(z) = \eta_2$ has exactly two solutions, 
which attract $z_+$ and $z_-$, respectively, if $\abs{\eta_2 - \eta_1}$ and 
hence $t_2$ are small enough.
Since $\eta_1$ is not a caustic point, all solutions of $f(z) = \eta_1$ are 
not in $\cC$.  Hence, every solution of $f(z) = \eta_2$ outside $D_\eps(z_0)$ 
attracts exactly one solution of $f(z) = \eta_1$ by Lemma~\ref{lem:transport1}, 
if $\abs{\eta_2 - \eta_1}$ is sufficiently small.
Together, $S$ is a prediction set of $f^{-1}(\{ \eta_2 \})$.

In the case of two fewer solutions, i.e., if $\arg(\eta_{2} - \eta_1) = \arg(-c)$,  
we remove two points from $f^{-1}( \{ \eta_1 \} )$ to obtain a prediction set.
To determine these points we consider the reversed step from $\eta_2$ to 
$\eta_1$.  This step gives the two additional solutions 
$\nm{z_\pm}{\eta_1}{f}{\infty}$, which we remove from $f^{-1}(\{ \eta_1 \})$ to 
get $S$.
Each solution of $f(z) = \eta_2$ attracts exactly one point in $S$, as above.
\end{proof}

\begin{remark} \label{rem:multiple_caustic_point}
Let $\eta \in f(\cC)$ be a (multiple) fold caustic point, i.e., for each $z_j 
\in \{ z_1, \ldots, z_m \} = f^{-1}(\{ \eta \}) \cap \cC$ the tangent $\tau_j$ 
from~\eqref{eqn:tangent_to_caus} is non-zero.
Then, for every $d \in \C$ with $\im(\conj{\tau}_j d) \neq 0$ for $j = 1, 
\ldots, m$, the effect of Theorem~\ref{thm:zero_at_crit} happens at all 
points $z_1, \ldots, z_m$ simultaneously.
By~\cite[Rem.~4.3]{SeteZur2019b}, $f(z) = \eta + t d$ has $2$ solutions in 
$D_\eps(z_j)$ and $f(z) = \eta - t d$ has no solutions in $D_\eps(z_j)$ 
if $\im(\conj{\tau}_j d) > 0$ and $t > 0$ is sufficiently small.
If $\im(\conj{\tau}_j d) < 0$ then $f(z) = \eta \pm t d$ swap their roles.
However, it is not guaranteed that the additional solutions attract the points 
$z_\pm(z_j)$ if $\arg(d) \neq \arg(\pm c(z_j))$, with $z_\pm$ and $c$ from 
Theorem~\ref{thm:zero_at_crit}.
\end{remark}

\begin{figure}
{\centering
\includegraphics[width=0.32\linewidth]{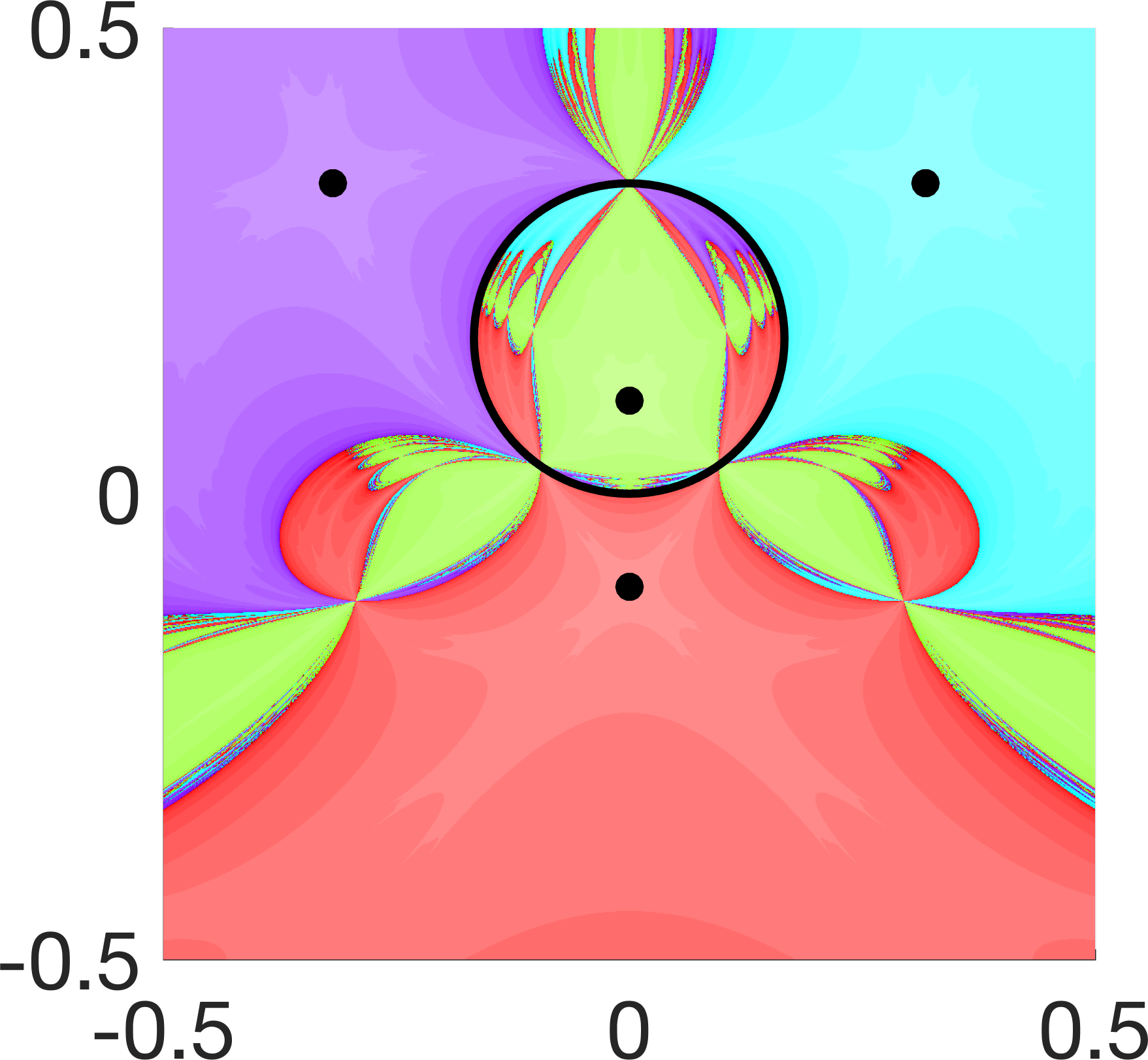}
\includegraphics[width=0.32\linewidth]{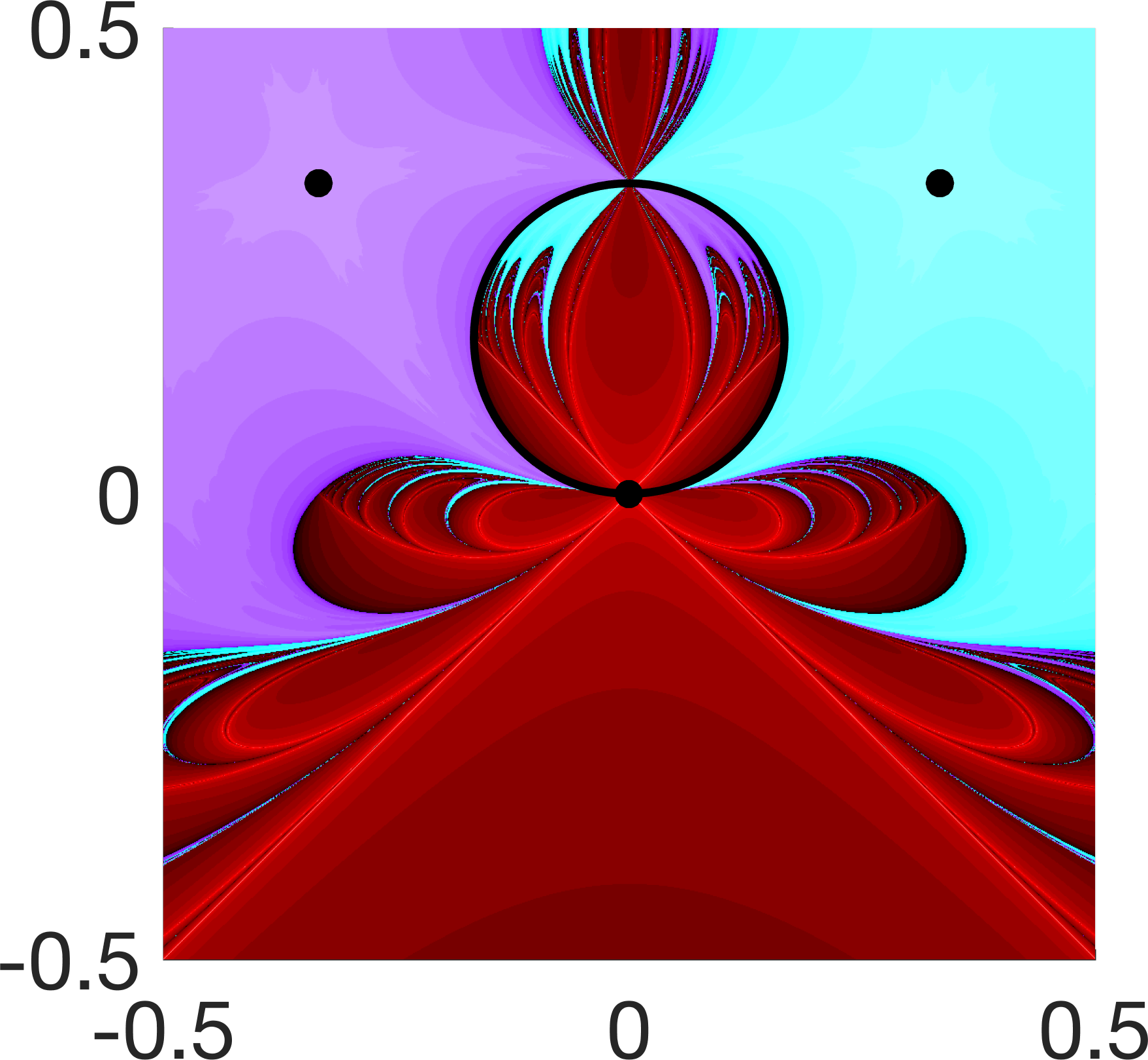}
\includegraphics[width=0.32\linewidth]{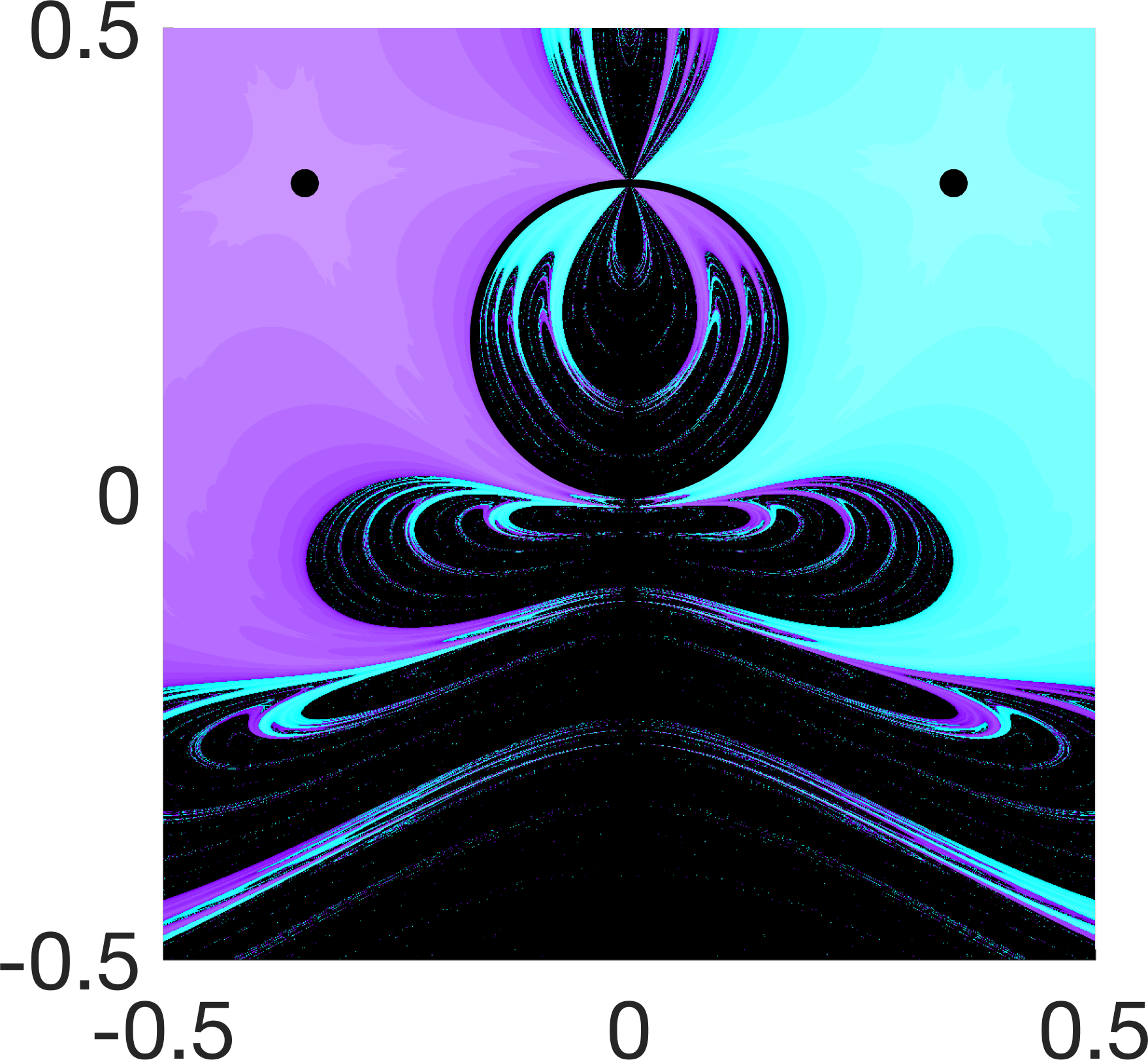}
}
\caption{
Solutions (black dots) of $f(z) = z + 2iz + \conj{z + iz^2} = t c$ with basins 
of attraction for $z_0 = 
0$, $c = -i$ and $t = 0.01$, $t = 0$, $t = -0.01$ (from left to right); see 
Remark~\ref{rem:basins}.  The black circle is the critical set of $f$.}
\label{fig:fractals_at_fold}
\end{figure}

\begin{remark} \label{rem:basins}
The basins of attraction of the two solutions of $f(z) = \eta + t c$ in 
$D_\eps(z_0)$ ($t > 0$) merge ($t = 0$) and disappear ($t < 0$) with the 
solutions.  This is illustrated in Figure~\ref{fig:fractals_at_fold}.
Points that are attracted by the same solution have the same color.
The darker the shading, the more iterations are needed until (numerical) 
convergence.
Points where the iteration does not converge (numerically) are colored in black.
This highlights why it is important to remove the two points 
$\nm{z_\pm}{\eta_1}{f}{\infty}$ from $f^{-1}( \{ \eta_1 \})$ 
in~\eqref{eqn:Sk_cross} for practical computations.
\end{remark}

\begin{figure}[t]
{\centering
\includegraphics[width=0.32\linewidth]{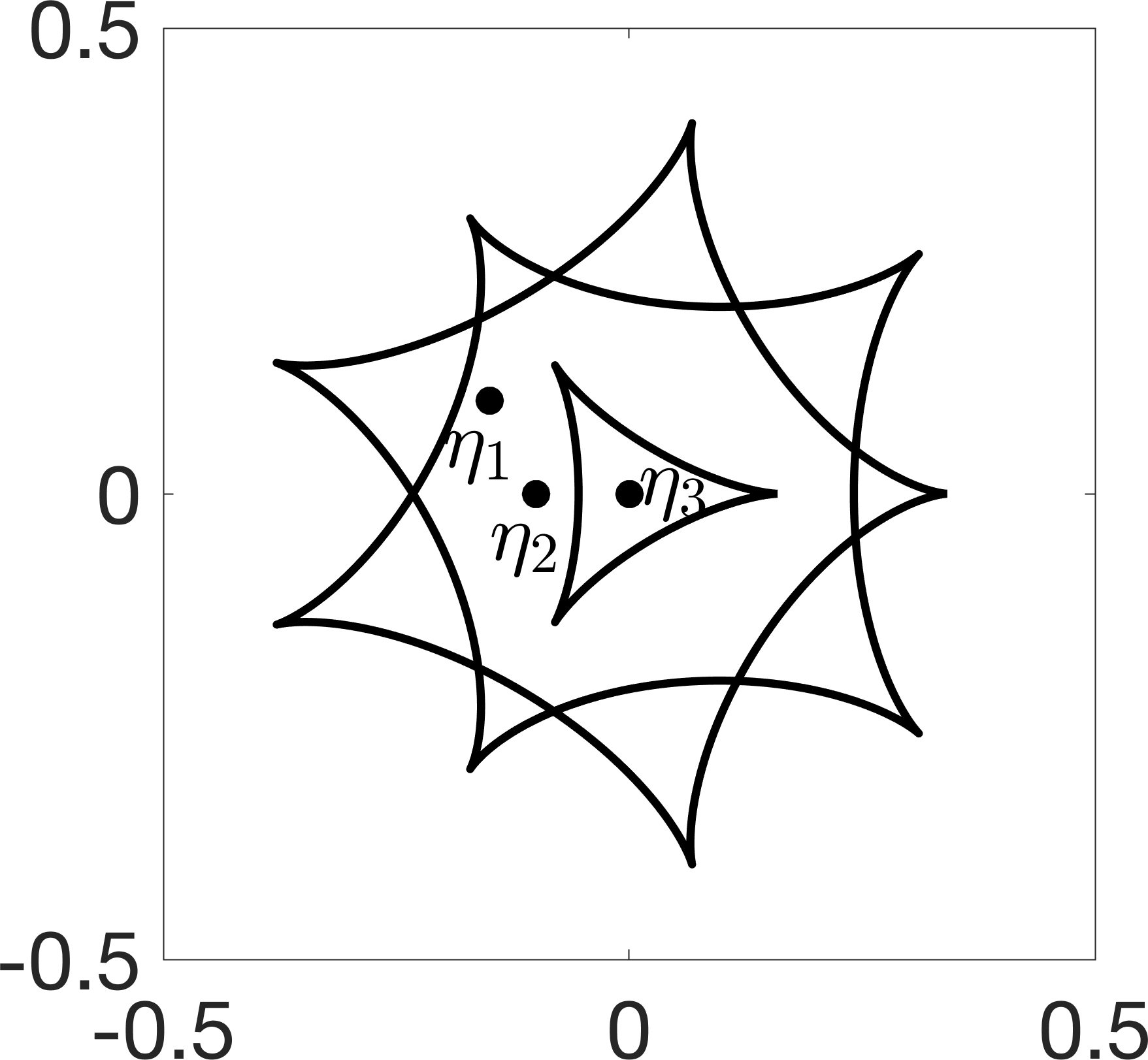}
\includegraphics[width=0.32\linewidth]{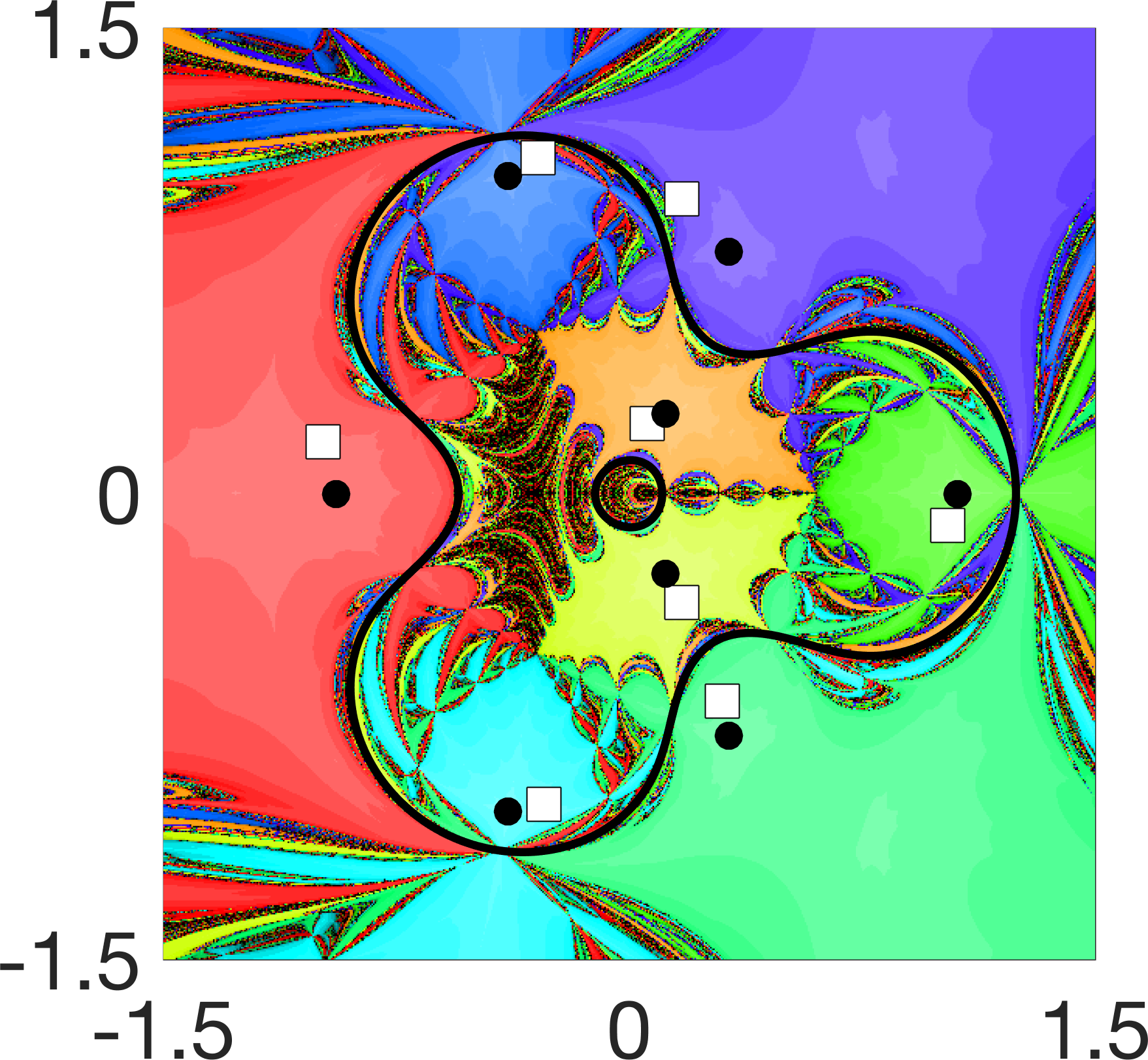}
\includegraphics[width=0.32\linewidth]{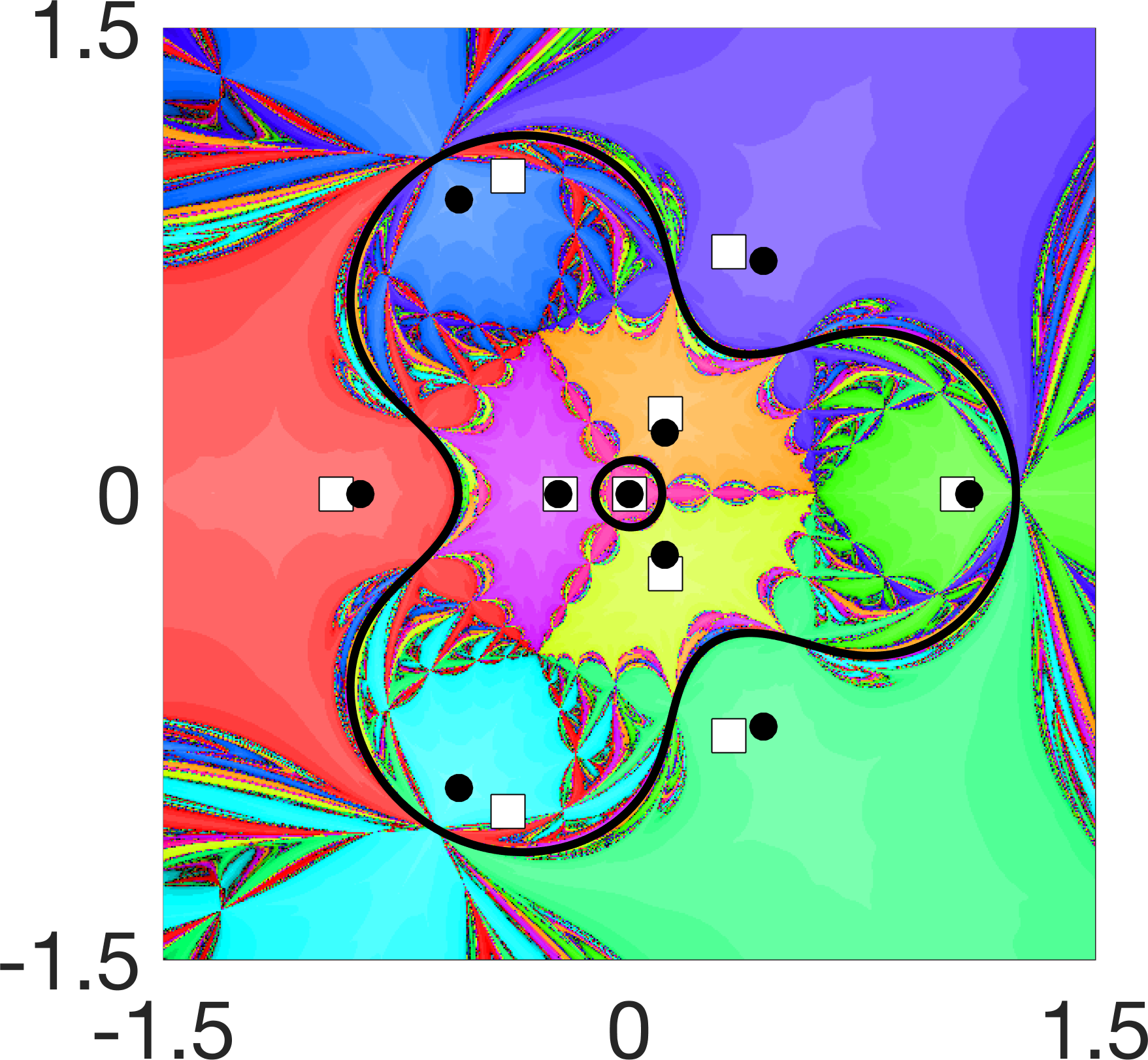}

}
\caption{Illustration of Theorems~\ref{thm:correction_regular} 
and~\ref{thm:correction_singular} for $f(z) = z - \conj{z^2/(z^3 - 0.6^3)}$ 
(see~\eqref{eqn:mpw} below).  Left: caustics and $\eta_1$, $\eta_2$, $\eta_3$.
Middle: prediction set $S= f^{-1}(\{\eta_1\})$ (white squares) of 
$f^{-1}(\{\eta_2\})$ (black dots).
Right: prediction set $S= f^{-1}(\{\eta_2\}) \cup \{z_\pm\}$ (white squares) of 
$f^{-1}(\{\eta_3\})$ (black dots).}
\label{fig:basins}
\end{figure}

We illustrate Theorems~\ref{thm:correction_regular} 
and~\ref{thm:correction_singular} in Figure~\ref{fig:basins}.
First, consider the step from $\eta_1$ to $\eta_2$ and the dynamics of 
$H_{f-\eta_2}$ (middle plot).  Every basin of attraction contains exactly one 
element of $S = f^{-1}(\{\eta_1\})$, i.e., $S$ is a prediction set of 
$f^{-1}(\{\eta_2\})$.
Some elements in $S$ are close to the boundary of the respective basins, such 
that $S$ may not be a prediction set of $f^{-1}(\{ \eta_2 \})$ if $\eta_2$ is slightly 
perturbed.
Since the step from $\eta_2$ to $\eta_3$ crosses a caustic, the number of 
solutions of $f(z) = \eta_2$ and $f(z) = \eta_3$ differ.
According to Theorem~\ref{thm:correction_singular} $S = f^{-1}(\{\eta_2\}) 
\cup \{ z_\pm \}$ with the points $z_\pm$ from~\eqref{eqn:zpm} is a prediction 
set of $f^{-1}(\{\eta_3\})$, shown in the right plot.

Theorems~\ref{thm:zeros_at_poles}, \ref{thm:correction_regular} 
and~\ref{thm:correction_singular} establish the transport of images method if 
the points $\eta_k$ are chosen properly.
This motivates the next definition.

\begin{definition}\label{def:transportpath}
Let $f$ be a non-degenerate harmonic mapping on $\widehat{\C}$.
We call $(\eta_1, \dots, \eta_n) \in \C^n$ a \emph{transport path} from 
$\eta_1$ to $\eta_n$ if for each $k = 1, \ldots, n-1$
either $f^{-1}(\{ \eta_k \})$ or $S$ in~\eqref{eqn:Sk_cross} is a prediction 
set of $f^{-1}(\{ \eta_{k+1} \})$.
\end{definition}

Along a transport path, all solutions of 
$f(z) = \eta_{k+1}$ can be computed from the solutions of $f(z) = \eta_k$.
Next we show the existence of such paths.
Here we call $\eta$ a multiple caustic point if $\abs{f^{-1}(\{ \eta \}) \cap 
\cC} > 1$.

\begin{theorem} \label{thm:transport_path}
Let $f$ be a non-degenerate harmonic mapping on $\widehat{\C}$ without 
non-isolated multiple caustic points, and let $\eta_s$, $\eta_e \in \C 
\setminus f(\cC)$.
Then there exists a transport path from $\eta_s$ to $\eta_e$.
\end{theorem}

\begin{proof}
We show that there exists a polygonal path from $\eta_s$ to $\eta_e$, 
which intersects the caustics only at a finite number of simple fold points.
Then this path is refined if necessary.

Let $R$ be the complex plane $\C$ without caustic points that are not simple 
fold points, which is $\C$ minus a set of isolated points, and hence a region 
(open and connected), as we discuss next.
Since $f$ is non-degenerate, the functions $\partial_z f$, $\conj{\partial_{\conj{z}} f}$
and $\omega$ are rational, see~\eqref{eqn:dzdzbarf}.
This has several implications.
The sets $\cM$ from~\eqref{eqn:def_cM} and $f(\cM)$ are finite.
The set $\cC \setminus \cM$ consists of $\deg(\omega)$ many pre-images 
of the unit circle under $\omega$.
There are at most finitely many critical points $z_0$ of $f$ with $\omega'(z_0) 
= 0$, hence the number of caustic points where the 
tangent~\eqref{eqn:tangent_to_caus} does not exist is finite.
The number of caustic points where the tangent exists and is zero is finite by 
Lemma~\ref{lem:tangent_to_caustic} (on each of the 
$\deg(\omega)$ many pre-images of the unit circle, there are at most finitely 
many points with $\tau = 0$, or the whole arc is mapped onto a single point).
By assumption, multiple caustic points are isolated.

Next, we prove the existence of a rectifiable path from $\eta_s$ to $\eta_e$ 
in $R$ with only finitely many caustic points (simple folds) on it.
Since $\eta_s, \eta_e \notin f(\cC)$ and since the caustics are compact there exists 
$\delta > 0$ such that $D_\delta(\eta_s)$, $D_\delta(\eta_e)$ contain no 
caustic points.
We consider the increasing sequence of regions $U_k \subseteq R$ whose points 
are `at most $k$ caustic crossings distant' from $\eta_s$, constructed as 
follows.
Let $U_0$ be the component (maximal open and connected set) of $R \setminus 
f(\cC)$ with $\eta_s \in U_0$.
Suppose that $U_k \subseteq R$ has been constructed.
Let $U_{k+1}$ be the component of $R \setminus (f(\cC) \setminus (\partial 
U_0 \cup \ldots \cup \partial U_k))$ with $\eta_s \in U_{k+1}$.
If $\eta_e \in U_k$, there exists a rectifiable path from $\eta_s$ to 
$\eta_e$ in $R$ that intersects $\partial U_0, \ldots, \partial U_{k-1}$, and 
hence has $k$ caustic crossings.
If $\eta_e \notin U_k$ then the boundary of $U_k$ in $R$, which consists of 
caustic arcs, has length at least $2 \pi \delta$, since then 
$D_\delta(\eta_s) \subseteq U_k$ and 
$D_\delta(\eta_e) \subseteq R \setminus U_k$.
Let $m = \ceil*{\frac{L}{2\pi\delta}}$, where $L < \infty$ is the total length 
of the caustics (see the discussion below Definition~\ref{def:nondegenerate}).
Then $\eta_e \in U_m$, and there exists a rectifiable path from $\eta_s$ to 
$\eta_e$ with at most 
$m$ caustic points.

By manipulations in an arbitrary small neighborhood around this path in~$R$ 
we obtain a polygonal path $P = (\eta_1, \ldots, \eta_n)$ 
with $\eta_1 = \eta_s$ and $\eta_n = \eta_e$, such that:
(1) $\eta_1, \ldots, \eta_n \in \C \setminus f(\cC)$, 
(2) each line segment $\cc{\eta_k, \eta_{k+1}}$ contains at most one caustic 
point 
and (3) if the line segment $\cc{\eta_k, \eta_{k+1}}$ contains a caustic point
then $\arg(\eta_{k+1} - \eta_k) = \arg(\pm c)$, with $c$ as in 
Theorem~\ref{thm:zero_at_crit}.

We refine this path to get a transport path.
First we consider the line segments $\cc{\eta_k, \eta_{k+1}}$, which contain a 
caustic point $\eta$, but where $S$ in~\eqref{eqn:Sk_cross} is not a prediction 
set of $f^{-1}(\{ \eta_{k+1} \})$.
We add two points from $\cc{\eta_k, 
\eta_{k+1}} \setminus f(\cC)$ sufficiently 
close to $\eta$ such that Theorem~\ref{thm:correction_singular} applies to 
these points.
Denote the refined path for simplicity again by $P = (\eta_1, \ldots, \eta_n)$.
Next we refine the line segments $E = \cc{\eta_k, \eta_{k+1}}$ without 
caustic points.
For all $\eta \in E$ there exists an $\eps(\eta) > 0$ such that 
$f^{-1}(\{\eta\})$ is a prediction set of $f^{-1}(\{\xi\})$ for all $\xi \in 
D_{\eps(\eta)}(\eta)$ by Theorem~\ref{thm:correction_regular}.
The family $(D_{\eps(\eta)}(\eta))_{\eta \in E}$ 
is an open covering of the compact set $E$.  Hence, there exists a finite 
subcovering  $(D_{\eps(\kappa_j)}(\kappa_j))_{j=1,\dots,\ell}$ with 
$\kappa_1 
= \eta_k$ and $\kappa_\ell = \eta_{k+1}$, which gives a partition of $E$, 
where $f^{-1}(\{ \kappa_j \})$ is a prediction set of $f^{-1}(\{ \kappa_{j+1} 
\})$ for $j = 1, \ldots, \ell-1$.
Refining all line segments without caustic points yields a transport path.
\end{proof}

Note that non-degenerate harmonic mappings can actually have non-isolated 
multiple caustic points.

\begin{example}
The harmonic mapping $f(z) = \frac{1}{2} (z^2 - 1)^2 + \conj{z^2 - 1}$ 
from~\cite[Ex.~5.1]{SeteZur2019b} is non-degenerate and maps its critical arcs 
$\gamma_\pm : \cc{-\pi, \pi} \to \cC$, 
$\gamma_\pm(t) = \pm \sqrt{1 + e^{-it}}$, onto the same caustic arc.
\end{example}

We close this section by noting the correctness of the transport of images 
method, provided that $0$ is not a caustic point of $f$. 

\begin{corollary}\label{cor:transport_path}
Let $f$ be a non-degenerate harmonic mapping on $\widehat{\C}$ without 
non-isolated multiple caustic points, such that $0$ is 
not a caustic point.  Then there exists a point $\eta_1 \in \C \setminus 
f(\cC)$,
such that Theorem~\ref{thm:zeros_at_poles} applies to all poles of $f$,
and there exists a transport path from $\eta_1$ to $0$. 
\end{corollary}

\begin{proof}
Theorem~\ref{thm:zeros_at_poles} applies to all poles of $f$ for all $\eta_1 \in \C$ 
with large enough $\abs{\eta_1}$.  Since $f(\cC)$ is compact, $\eta_1$ can be 
chosen in $\C \setminus f(\cC)$.  Then there exists a transport path from 
$\eta_1$ to $0$ by Theorem~\ref{thm:transport_path}.
\end{proof}

\begin{remark}
In Corollary~\ref{cor:transport_path}, the point $0$ is not a caustic 
point of $f$.  If~$0$ is a caustic point then $f$ has (at least) one singular 
zero $z_0 \in \cC$.
Since continuation may run into numerical trouble if solutions are almost 
singular,
special strategies are used.
These are usually referred to as the \emph{endgame}.  For systems of analytic 
functions, two commonly used endgame strategies are the \emph{Cauchy 
endgame}~\cite{MorganSommeseWampler1991} and the 
\emph{power series endgame}~\cite{MorganSommeseWampler1992}.
The setting for the transport of images method is somewhat different since we 
require the critical set $\cC$ in the method.  Therefore, we can 
(numerically) compute the zeros of $f$ in $\cC$, which are exactly the singular 
zeros.  Thus, in the endgame we would only have to determine which homotopy curves 
intersect at $z_0$.
This can be done with (2) and (3a) in Section~\ref{sect:homotopycurves} if 
$f$ is non-degenerate and $\omega'(z_0) \neq 0$.
\end{remark}

\subsection{Analysis of the homotopy curves}
\label{sect:homotopycurves}

We now consider the continuous problem behind our (discrete) computation and
describe the homotopy curves, i.e.,
how the solution set of $f(z) = \eta(t)$ varies with $t$, where $f : \Omega 
\to \C$ is a harmonic mapping and $\eta : \cc{a, b} \to \C$, $t \mapsto 
\eta(t)$, is a 
(continuous) path.
We analyze the solution set locally with results 
from~\cite{Lyzzaik1992,SeteZur2020,SeteZur2019b}.
Combining the local results gives the global picture.
We distinguish the cases (1) $z_0 \in \Omega \setminus \cC$, (2) $z_0 \in \cM$ 
and (3) $z_0 \in \cC \setminus \cM$.

\textbf{(1)} Let $z_0 \in \Omega \setminus \cC$, i.e., the Jacobian of $f$ is 
non-zero at $z_0$.  By the inverse function theorem there exists an $\eps > 0$ 
and an open set $V \subseteq \C$ such that $f : D_\eps(z_0) \to V$ is a 
diffeomorphism.  Thus, a curve $\eta(t)$ in $V$ has a unique pre-image curve $z(t) = 
f^{-1}(\eta(t))$.
In particular, the homotopy curves can intersect only at critical points.

\textbf{(2)} Let $z_0 \in \cM$ from~\eqref{eqn:def_cM}.
Then $f$ has the form
\begin{equation*}
f(z) = a_0 + \sum_{k=n}^\infty a_k (z-z_0)^k + \conj{b_0 + \sum_{k=n}^\infty 
b_k (z-z_0)^k}
\end{equation*}
with $n \geq 2$ and $\abs{a_n} \neq \abs{b_n}$.  If $\eta(t)$ passes through 
$f(z_0)$  then exactly $n$ homotopy curves intersect at $z_0$ with equispaced 
angles, since $f(z) = \eta(t)$ is locally to leading order $f(z_0) + a_n 
(z-z_0)^n + \conj{b_n (z-z_0)^n} = \eta(t)$.
The latter can be uniquely solved for $(z-z_0)^n$ by
\begin{equation*}
(z - z_0)^n = \frac{\conj{a}_n (\eta(t) - f(z_0)) - \conj{b_n (\eta(t) - 
f(z_0))}}{\abs{a_n}^2 - \abs{b_n}^2},
\end{equation*}
see~\cite[Lem.~4.2]{SeteZur2020}, which yields the $n$ solutions.

\textbf{(3)} For $z_0 \in \cC \setminus \cM$ the situation is more involved.
By~\cite[Thm.~2.1]{Lyzzaik1992}, a harmonic mapping is either (a) light,
(b) has zero Jacobian or (c) is constant on some arc of $\cC \setminus \cM$.
We discuss these three cases.

\textbf{(3a)} Let $f$ be a light harmonic mapping.
Essentially, the behavior of the homotopy curves at $z_0$ depends on the 
tangent to the caustics at $f(z_0)$ and hence on $\psi$ 
from~\eqref{eqn:psi}.
Let $z_0 \in \cC \setminus \cM$.  Then $z_0$ is a zero of order $\ell \geq 0$
of $\partial_z f$ (and of $\conj{\partial_{\conj{z}} f}$).  First let 
$\omega'(z_0) \neq 0$, so that $z_0 = \gamma(t_0)$ is a 
point on a critical curve that is not a self-intersection 
point of the critical curves.
\begin{enumerate}
\item If $\psi(t_0)$ is non-zero ($f(z_0)$ is a fold caustic point and consequently
$\ell = 0$) then $z_0$ is a turning point of the homotopy curves, i.e., two 
curves start or end at $z_0$ if $\eta(t)$ crosses the caustics at $f(z_0)$.
This follows from Theorem~\ref{thm:zero_at_crit} 
or~\cite[Thm.~5.1]{Lyzzaik1992} and is illustrated 
in~\cite[Fig.~5]{SeteZur2019b}.

\item More generally, if $\psi$ does not change sign at $t_0$ 
then~\cite[Thm.~5.1]{Lyzzaik1992} implies that $\ell$ homotopy curves intersect 
at $z_0$ and two homotopy 
curves start or end at $z_0$, if $\eta(t)$ crosses the caustics at $f(z_0)$. 
Hence, $z_0$ is a turning point and, if $\ell > 0$, a bifurcation point of 
the homotopy curves.

\item If $\psi$ changes sign at $t_0$ (i.e., $f(z_0)$ is a cusp)
then by~\cite[Thm.~5.1]{Lyzzaik1992} $z_0$ is a bifurcation point where
$\ell+1$ homotopy curves intersect and two homotopy curves start or end, if 
$\eta(t)$ crosses the caustics at $f(z_0)$.
\end{enumerate}
For $\omega'(z_0) = 0$ there is only an upper bound on the local valence of 
$f$ at $z_0$ in~\cite[Sect.~6]{Lyzzaik1992} and hence on the number of homotopy 
curves, which also depends on the caustic tiles bordering $f(z_0)$ that are
traversed by $\eta(t)$.

\begin{figure}
{\centering
\includegraphics[width=0.43\linewidth]{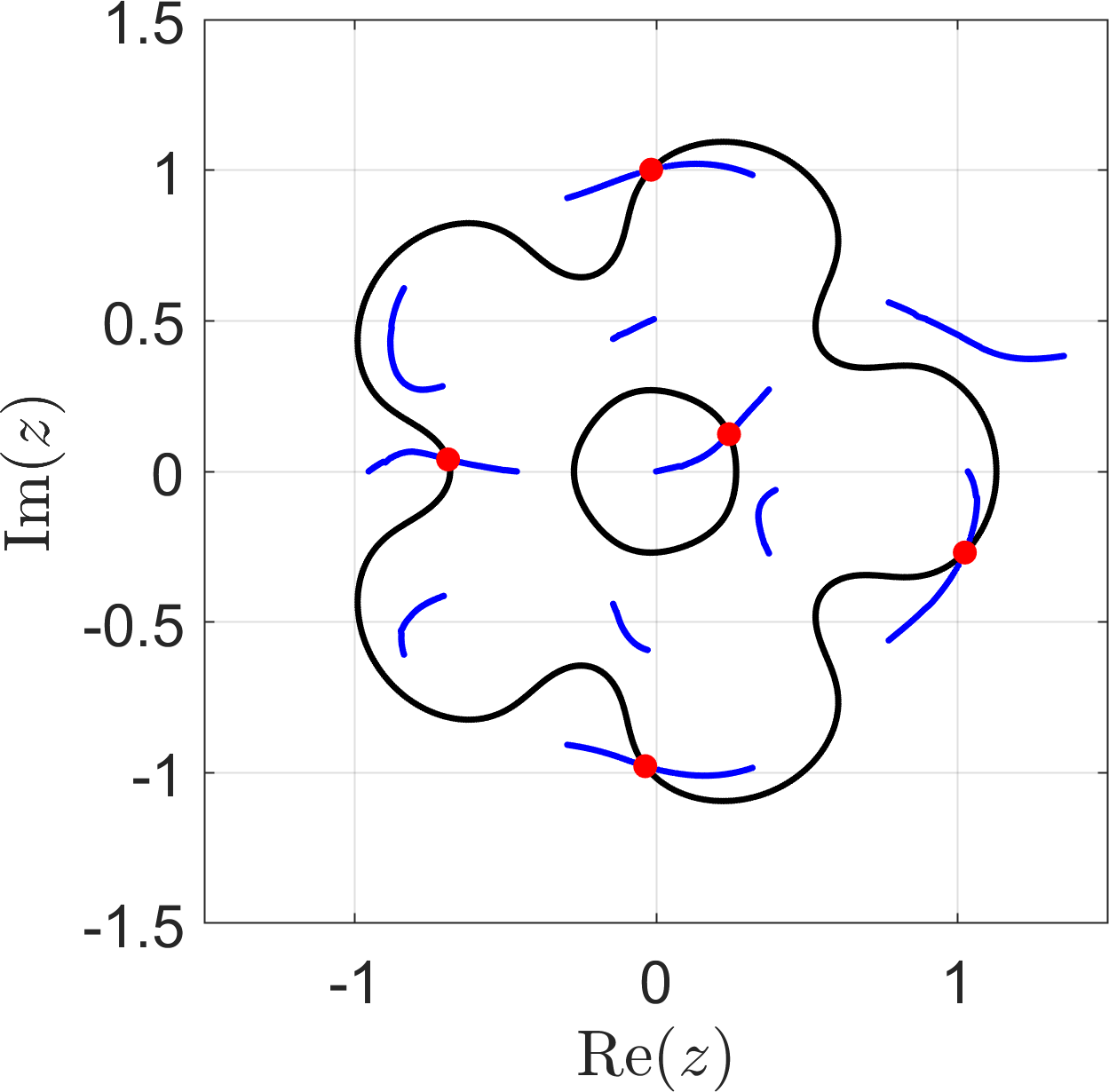}
\hfill
\includegraphics[width=0.55\linewidth]{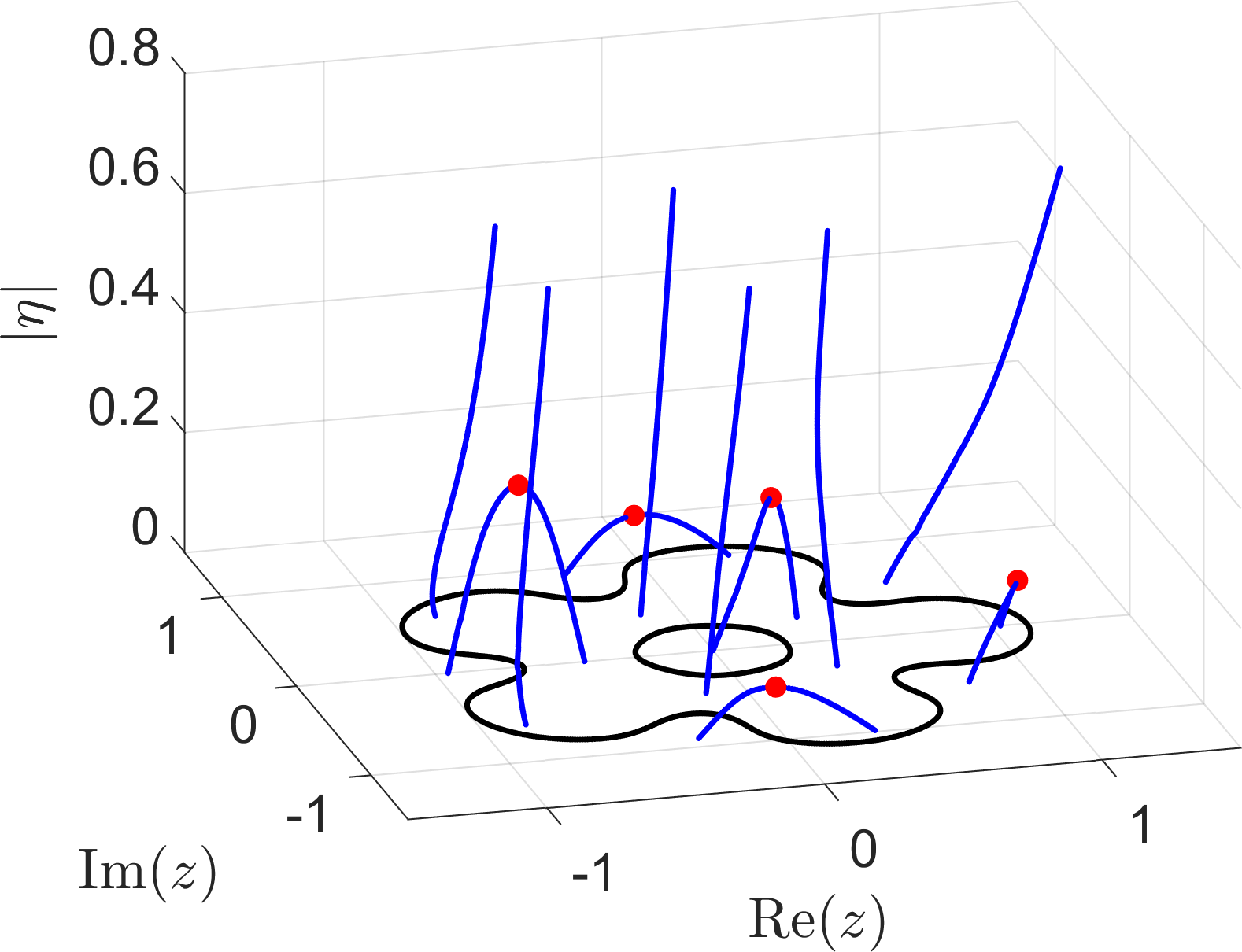}

}
\caption{Computed homotopy curves (blue) with turning points (red dots) in 
Example~\ref{ex:homotopycurves} and critical set of $f$ (black).}
\label{fig:homotopycurves}
\end{figure}

\begin{example} \label{ex:homotopycurves}
Consider the harmonic mapping~\eqref{eqn:mpw} with $n=5$ and $r = 0.6$, which is light and 
non-degen\-er\-ate. 
Figure~\ref{fig:homotopycurves} shows the computed homotopy curves where $\eta = 
\eta(t)$ 
parametrizes the segment from $0.6699 + 0.1795i$ to $0$, which intersects the 
caustics of $f$ only at simple fold points.  Along this path each caustic 
crossing 
produces two additional solutions.  The corresponding turning points of the 
homotopy curves are marked with red dots.
Figure~\ref{fig:homotopycurves} shows the solutions $z(t)$ of $f(z) = \eta(t)$ 
in the $z$-plane (left) and 
over the $z$-plane at height $\abs{\eta(t)}$ (right).
\end{example}

\textbf{(3b)} Let $f$ be a harmonic mapping with $J_f \equiv 0$, i.e., $\cC = 
\Omega$.
Then $f$ has the form $f(z) = a \re(h(z)) + b$ with $a, b \in \C$, $a \neq 0$, and an analytic 
function $h$; see~\cite[Lem.~2.1]{Lyzzaik1992} or~\cite[Lem.~4.7]{Neumann2005}.
(If $\Omega$ is not simply connected, $h$ may be multi-valued, but $\re(h)$ is 
single-valued.)
Then $f(\Omega) \subseteq \{ a t + b : t \in \R \}$, and $f(z) = \eta$ if and 
only if $\re(h(z)) = (\eta - b)/a$.  Since $h$ is analytic, it is either 
constant or has the open mapping property, which shows that the solution set of 
$f(z) = \eta$ is either empty, $\Omega$ or a union of curves.
Therefore, along a path $\eta(t)$ crossing the caustics of $f$, the solutions do not form 
curves depending on $t$.

\begin{example}
The Jacobian of $f(z) = z + \conj{z}$ vanishes 
identically.  Clearly, $f(z) = \eta \in \C \setminus \R$ has no solution, while 
$f(z) = \eta \in \R$ is solved by all $z$ with $\re(z) = \eta/2$.
\end{example}

\textbf{(3c)} Let $f$ be a harmonic mapping that is constant on some arc 
$\Gamma$ of 
$\cC \setminus \cM$, and $z_0 \in \Gamma$.  Then $f$ maps $\Gamma$ onto the 
single point $f(z_0)$.
For $\eta(t)$ not equal to $f(z_0)$ (or any similar point), the homotopy curves 
are as for light harmonic mappings.  However, if $\eta(t) = f(z_0)$ then
the solution set also contains all points of $\Gamma$, and the solutions do not 
form curves depending on~$t$.

\begin{example} \label{ex:chang_refsdal}
The non-degenerate harmonic mapping $f(z) = z - 1/\conj{z}$ is not light since
$f^{-1}( \{ 0 \}) = \{ z \in \C : \abs{z} = 1 \} = \cC$.
Let $\eta \in \C \setminus \{ 0 \}$.  Then $f(z) = \eta$ has the two solutions
\begin{equation*}
z_\pm(\eta)
= \frac{1 \pm \sqrt{1 + 4 \abs{\eta}^{-2}}}{2} \eta
\end{equation*}
which satisfy $\abs{z_+(\eta)} > 1$ and $\abs{z_-(\eta)} < 1$, and 
$\abs{z_\pm(\eta)} \to 1$ if $\eta \to 0$.
Figure~\ref{fig:chang_refsdal} shows a path $\eta(t)$ crossing the origin 
(panel 1) and the corresponding homotopy curves (panel 2).
For $\eta(t) = 0$ the two solutions are replaced by all points on the unit 
circle.
Panels 3 and 4 show a similar example with a non-smooth path $\eta(t)$, 
which leads to homotopy curves exiting $\cC$ at different points.
\end{example}

\begin{figure}
{\centering
\includegraphics[width=0.24\linewidth]{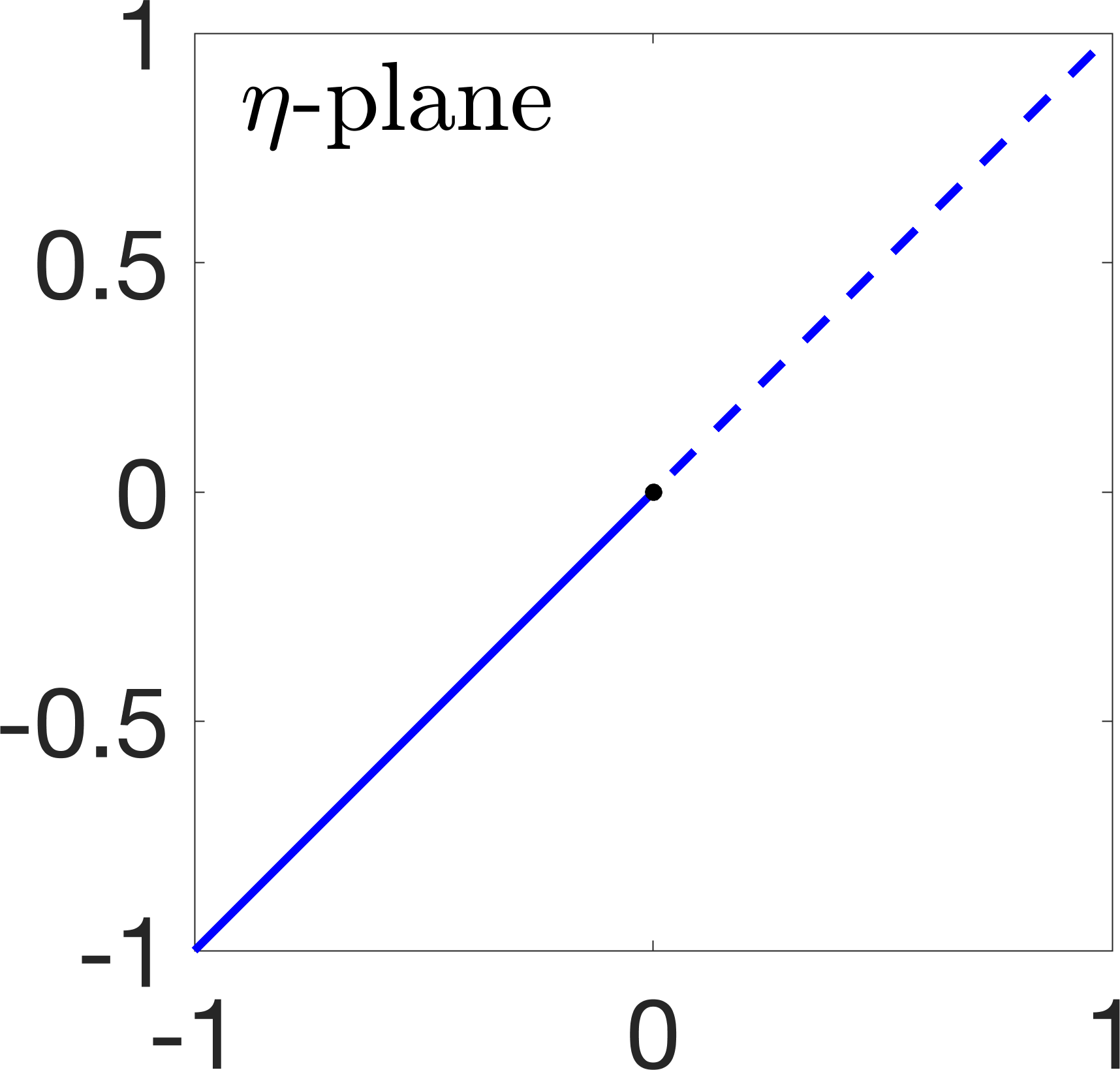}
\includegraphics[width=0.24\linewidth]{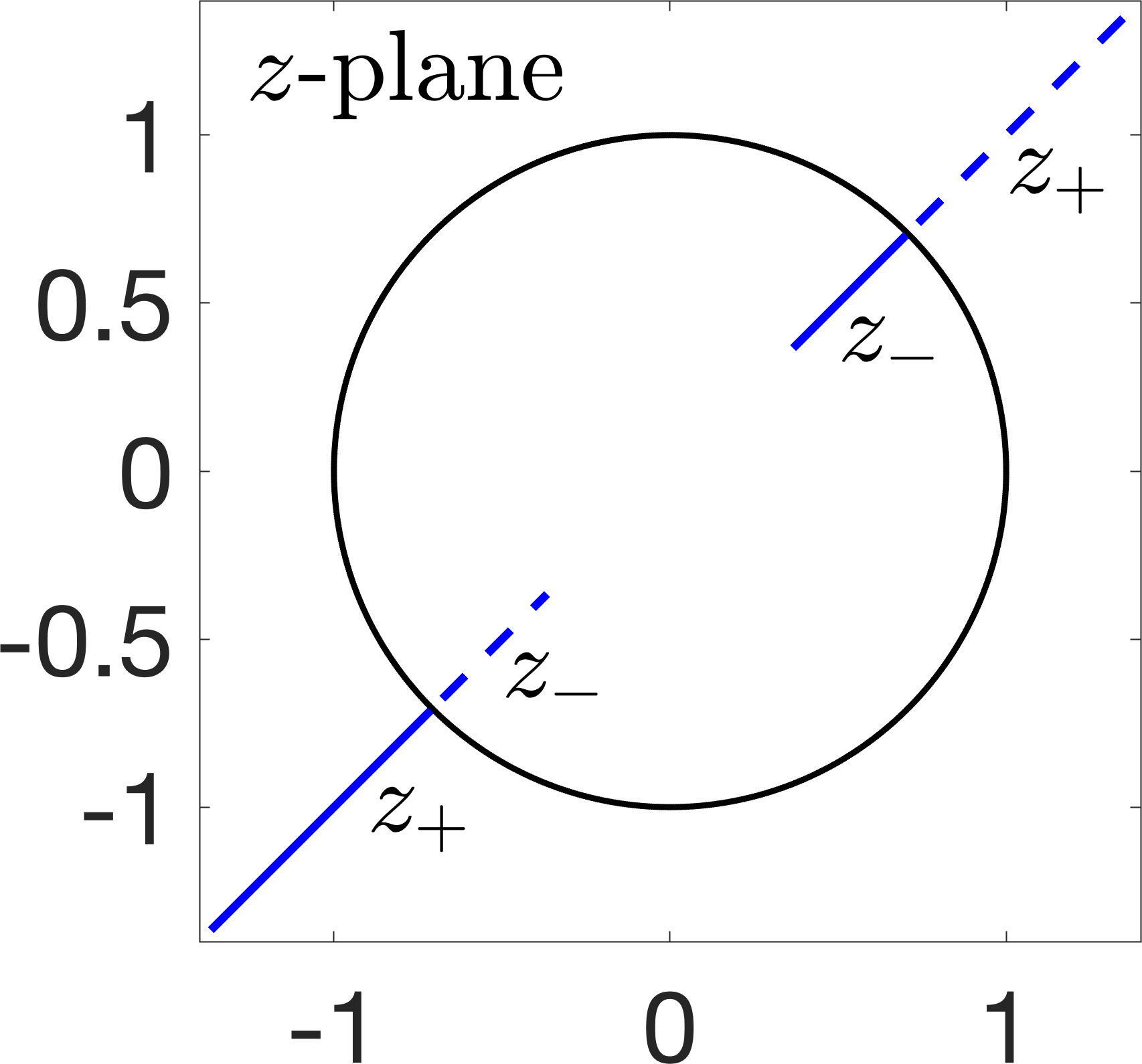}
\includegraphics[width=0.24\linewidth]{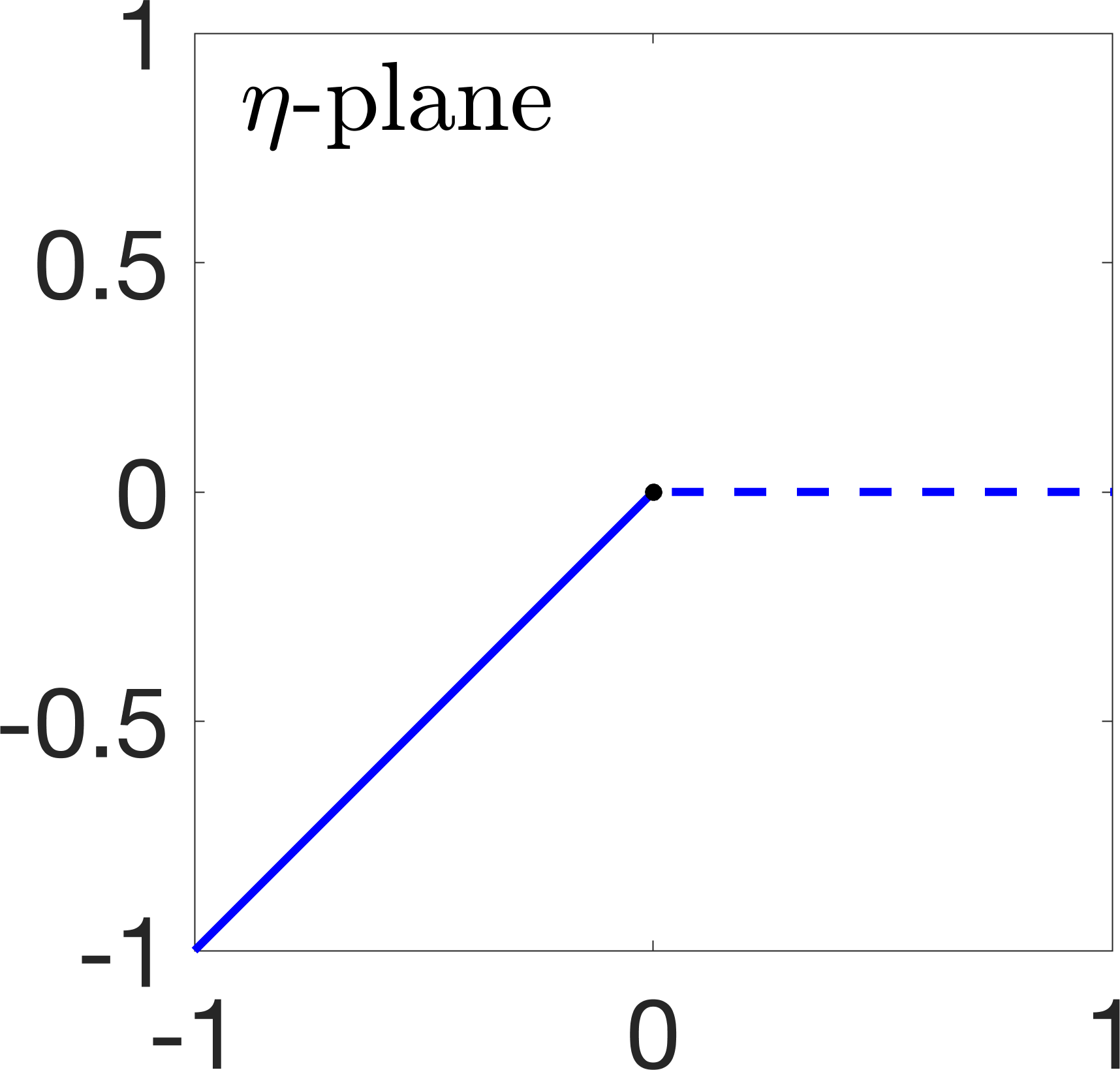}
\includegraphics[width=0.24\linewidth]{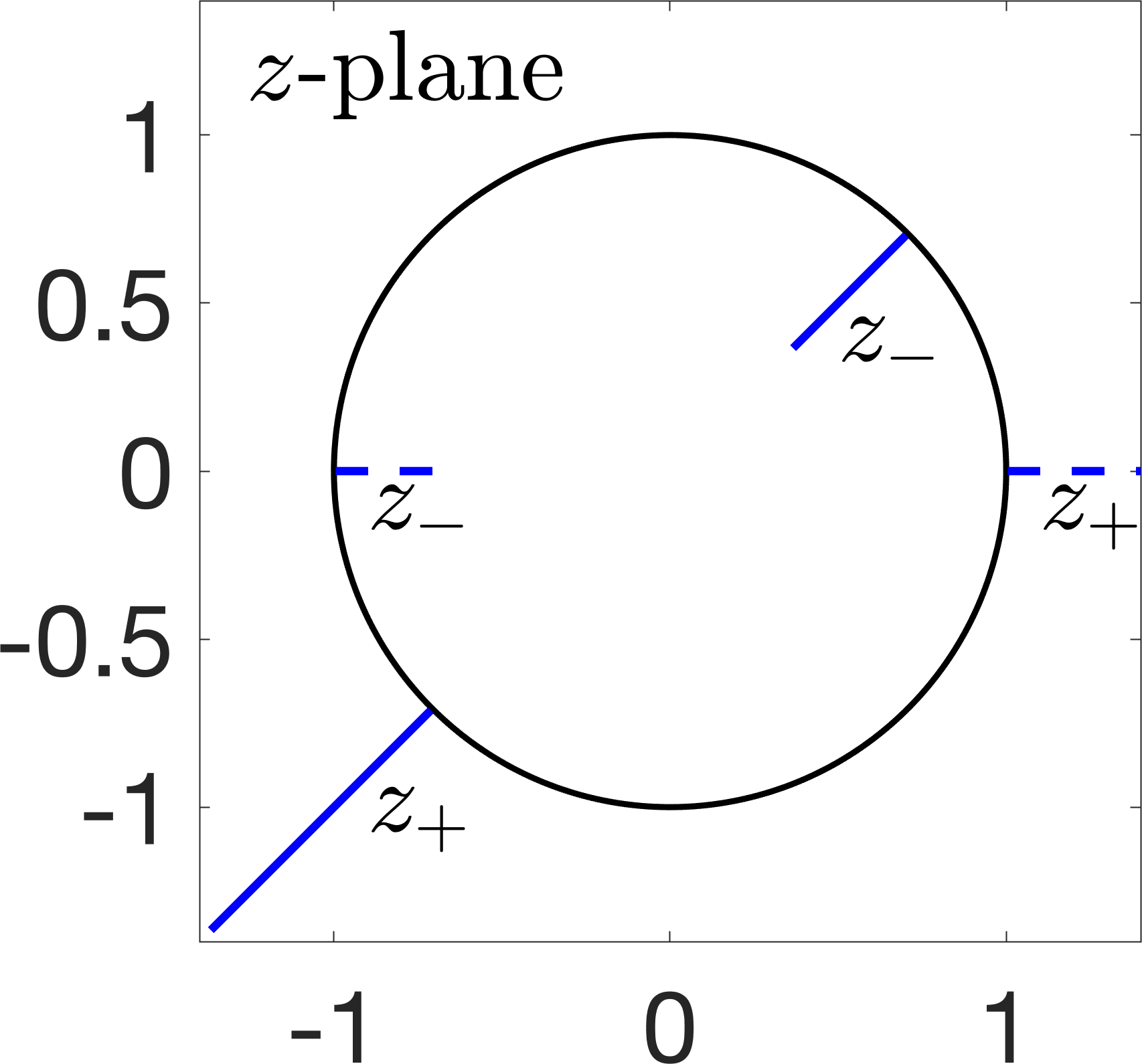}

}
\caption{Homotopy curves of $f(z) = z - 1/\conj{z}$ for two paths $\eta(t)$ 
(panels 1 and 3), divided in a solid and dashed blue part, 
and corresponding pre-images (panels 2 and 4); see 
Example~\ref{ex:chang_refsdal}.
The black circle is $\cC = f^{-1}(\{ 0 \})$.}
\label{fig:chang_refsdal}
\end{figure}

The analysis of the homotopy curves also confirms that 
Definition~\ref{def:transportpath} of a transport path is reasonable.
The points $\eta_k$ are in $\C \setminus f(\cC)$ to avoid a zero Jacobian 
and potentially infinitely many solutions.
Moreover, to avoid bifurcation points of the homotopy curves, 
transport paths cross the caustics only at fold points, which can be handled 
by Theorem~\ref{thm:zero_at_crit}.

\section{Implementation}
\label{sect:implementation}

A MATLAB implementation of our method is freely available at
\begin{center}
\url{https://github.com/transportofimages/}
\end{center}
This toolbox contains routines to compute the zeros of non-degenerate harmonic 
mappings and the critical curves and caustics.
Moreover, it contains m-files to reproduce all examples in 
Section~\ref{sect:numerics}.
We briefly describe two key aspects of our implementation.

\subsection{Computation of the caustics}\label{sect:caustic_computation}

The transport of images methods bases essentially upon the correct handling of 
caustic crossings.
To compute the caustics we first compute the critical curves.
Critical points $z \in \cC \setminus \cM$ satisfy $\omega(z) = e^{i t}$ 
with $t \in \co{0, 2 \pi}$; see~\eqref{eqn:crit_omega}.
Since $\omega$ is rational for non-degenerate harmonic mappings, $\omega(z) = 
e^{it}$ is equivalent to a polynomial equation with $\deg(\omega)$ many 
solutions.
We first solve $\omega(z) = 1$ via the polynomial equation.  Then,  for
$0 = t_1 < t_2 < \ldots < t_k < 2 \pi$ we 
successively solve $\omega(z) = e^{i t_{j+1}}$ with Newton's method
(with initial points $\omega^{-1}(\{ e^{i t_j} \})$).
If successful,
this procedure gives $\deg(\omega)$ many (discretized) critical arcs, 
parametrized according to~\eqref{eqn:parametrization}.
Gluing them together yields the critical curves.
Finally, the image under $f$ of the critical curves is a 
discretization of the caustics of $f$.

\subsection{Construction of transport paths}
\label{sect:transport_path}

We construct a transport path to $0$
along $R_{\theta} = \{re^{i\theta}: r > 0\}$, $\theta \in \co{0, 2 \pi}$. 
Beneficially, the intersection points of the caustic and the ray $R_\theta$
can be read off from the argument of the caustic points.
Also, as discussed in Remark~\ref{rem:multiple_caustic_point}, crossing the 
caustics in a specific direction is not essential, as long as the crossing 
direction is not tangential to the caustics.  For brevity we still 
call $(\eta_1,\dots,\eta_n) \in \C^n$ a transport path if $f^{-1}(\{ \eta_k 
\})$ or $S$ in~\eqref{eqn:Sk_cross} is a prediction set of $f^{-1}(\{ 
\eta_{k+1} \})$, even if the step from $\eta_k$ to $\eta_{k+1}$ is not in the 
direction $c$ in Theorem~\ref{thm:correction_singular}, i.e., even if 
$\arg(\eta_{k+1} - \eta_k) \neq \arg(\pm c)$.
Determining a suitable angle $\theta \in \co{0, 2\pi}$ for $R_\theta$ can be 
challenging.  In particular, the transport 
phase (along $R_\theta$) may fail if too small step sizes are required,
or if $R_\theta$ is tangential or almost tangential to the caustics,
or if $R_\theta$ intersects the caustic at or near cusps or other 
non-fold points.  To overcome this difficulty we choose the angle $\theta \in 
\co{0, 2\pi}$ at random (e.g., uniformly distributed). If the transport phase 
with $\theta$ is not successful we restart with a new random angle.

We construct a transport path to $0$ along $R_\theta$ as follows:

\textbf{(1)} Set $\eta_1 = 2e^{i \theta} \max_{z \in \cC} \abs{f(z)}$.  If 
Theorem~\ref{thm:zeros_at_poles} with $\eta_1$ does not apply to all poles of 
$f$, e.g., two distinct points $\zeta_j$ and $\zeta_\ell$ are attracted by the 
same solution of $f(z) = \eta_1$, we increase $\abs{\eta_1}$ and proceed.  
Eventually, $\abs{\eta_1}$ is large enough and Theorem~\ref{thm:zeros_at_poles} 
applies to all poles of $f$ simultaneously.

\textbf{(2)} Let $\{ \xi_1, \dots, \xi_\ell \} = f(\cC) \cap R_\theta$ with 
$\abs{\xi_1} > \ldots > \abs{\xi_\ell} > 0$.
We compute these intersection points from the discretized caustics with 
bisection.
For each $\xi_k$ let $\eta_{2k}, \eta_{2k+1} \in R_\theta$ with equal distance 
to $\xi_k$, and such that $\abs{\eta_1} > \abs{\eta_2} > \ldots > 
\abs{\eta_{2\ell + 1}} > 0$.  Then $P = (\eta_1,\eta_2,\dots,\eta_{2\ell + 1}, 
0)$ is a potential transport path.

\textbf{(3)} We refine $P$ recursively with the following 
divide-and-conquer scheme to obtain a transport path.

\textbf{(3a)}  If the step from $\eta_j$ to $\eta_{j+1}$, where the line 
segment $\cc{\eta_j,\eta_{j+1}}$ does not intersect the caustics, is 
unsuccessful, i.e., if $f^{-1}(\{\eta_j\})$ is not a prediction set of 
$f^{-1}(\{\eta_{j+1}\})$, we divide the step from $\eta_j$ to $\eta_{j+1}$ into 
two substeps by inserting $\eta' = (\eta_j + \eta_{j+1})/2$ into $P$. 

\textbf{(3b)} If the step from $\eta_j$ to $\eta_{j+1}$, where the line segment 
$\cc{\eta_j,\eta_{j+1}}$ intersects the caustic, is unsuccessful, i.e., if $S$ 
in~\eqref{eqn:Sk_cross} is not a prediction set of $f^{-1}(\{\eta_{j+1}\})$, we 
divide the step from $\eta_j$ to $\eta_{j+1}$ into three substeps by inserting 
$\eta_1' = (3 \eta_j + \eta_{j+1})/4$ and $\eta_2' = (\eta_j + 3 \eta_{j+1})/4$ 
into $P$.

\textbf{(4)} We proceed, dividing the steps in (3a) and (3b) until $P$ is a 
transport path, or until we reached a prescribed recursion depth. 
In the latter case we restart with a new random angle $\theta$ in (1).

\section{Numerical examples}
\label{sect:numerics}

In order to verify the numerical results we test our method on the following
functions, for which the number of zeros is known analytically.
\begin{enumerate}
\item Wilmshurst's harmonic polynomial
\begin{equation}\label{eqn:wilmshurst}
f(z) = (z-1)^n + z^n + \conj{i (z-1)^n - i z^n}, \quad n \geq 1,
\end{equation}
is non-degenerate and has exactly $n^2$ zeros.  This is the maximum number of 
zeros of a harmonic polynomial $f(z) = p(z) + \conj{q(z)}$ with $\deg(p) = n$ 
and $\deg(q) < n$; see~\cite{Wilmshurst1998}.

\item The rational harmonic mapping in Mao, Petters and 
Witt~\cite{MaoPettersWitt1999},
\begin{equation}\label{eqn:mpw}
f(z)= z - \conj{\left(\frac{z^{n-1}}{z^n - \rho^n}\right)},
\quad n \geq 3, \quad \rho > 0,
\end{equation}
is non-degenerate and has exactly $3n+1$ zeros for $\rho < \rho_{c}$, 
$2n+1$ zeros for $\rho = \rho_{c}$ and $n+1$ zeros for $\rho > \rho_{c}$, 
where 
\begin{equation*}
\rho_{c} = \left( \frac{n-2}{n} \right)^{\frac{1}{2}} 
\left(\frac{2}{n-2}\right)^{\frac{1}{n}};
\end{equation*}
see~\cite[Prop.~2.1]{LuceSeteLiesen2014}.

\item Rhie's rational harmonic mapping
\begin{equation}\label{eqn:rhie}
f(z)
= z - \conj{\left((1-\eps) \frac{z^{n-1}}{z^n - \rho^n} + 
\frac{\eps}{z}\right)},
\quad n \geq 3, \quad \rho, \eps > 0,
\end{equation}
is non-degenerate and has exactly $5n$ zeros if $\rho < \rho_{c}$ (as above) 
and if $\eps < \eps^*$, where $\eps^*$ depends on $n$ and $\rho$; 
see~\cite[Thm.~2.2]{LuceSeteLiesen2014} for details.
This is the maximum number of zeros of a harmonic mapping $f(z) = z - 
\conj{r(z)}$ with a rational function $r$ of degree $n+1 \geq 2$; 
see~\cite[Thm.~1]{KhavinsonNeumann2006}.
\end{enumerate}

These functions should be challenging for the transport of images method
since they have a large number of zeros and very nested caustics; see 
Figure~\ref{fig:exp_caus}.
In gravitational lensing, the functions~\eqref{eqn:mpw}, \eqref{eqn:rhie} and 
their zeros are of interest; see~\cite{KhavinsonNeumann2008, 
LuceSeteLiesen2014}.

\begin{figure}[t]
\includegraphics[width=0.308\linewidth]{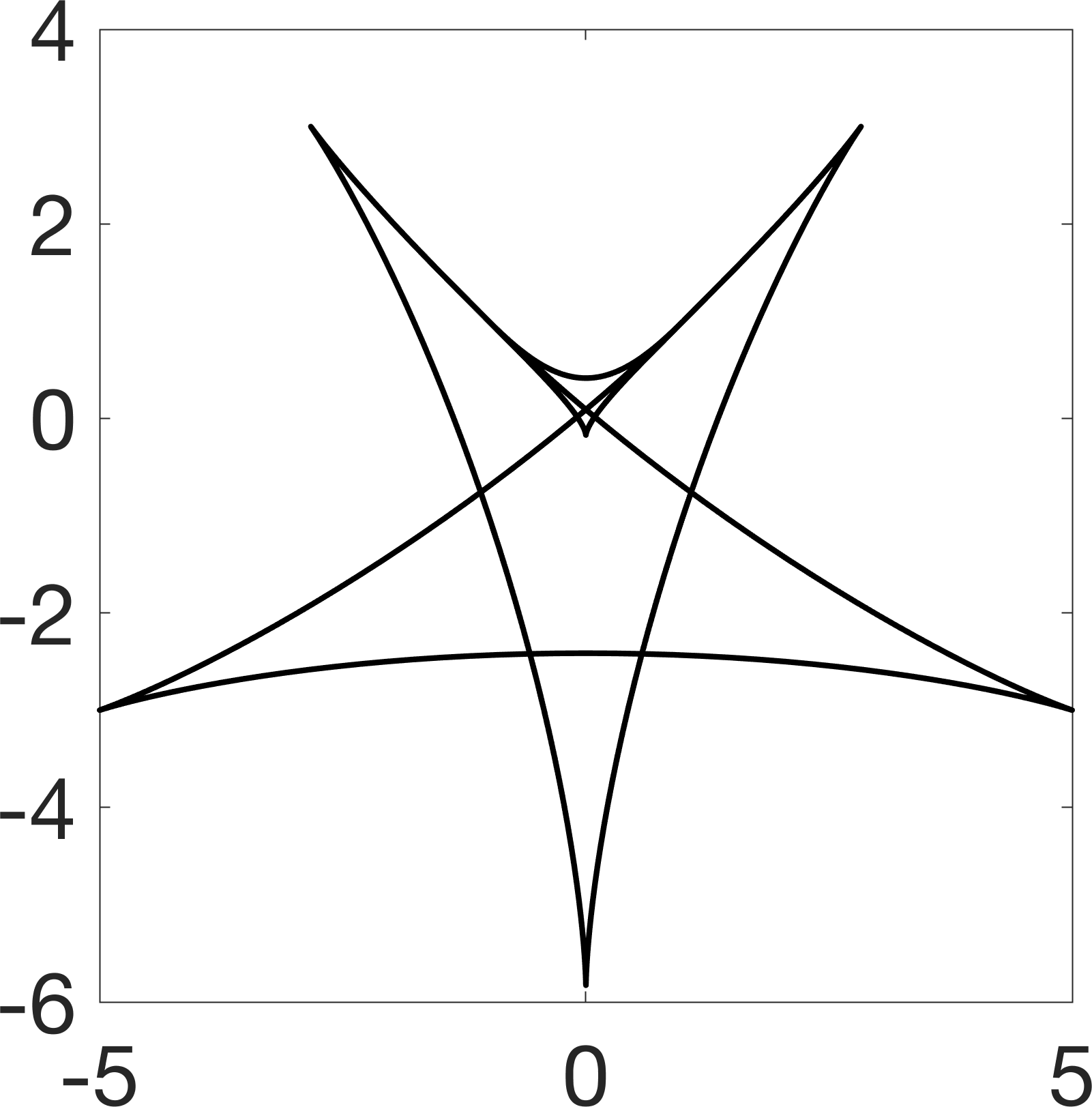}
\includegraphics[width=0.338\linewidth]{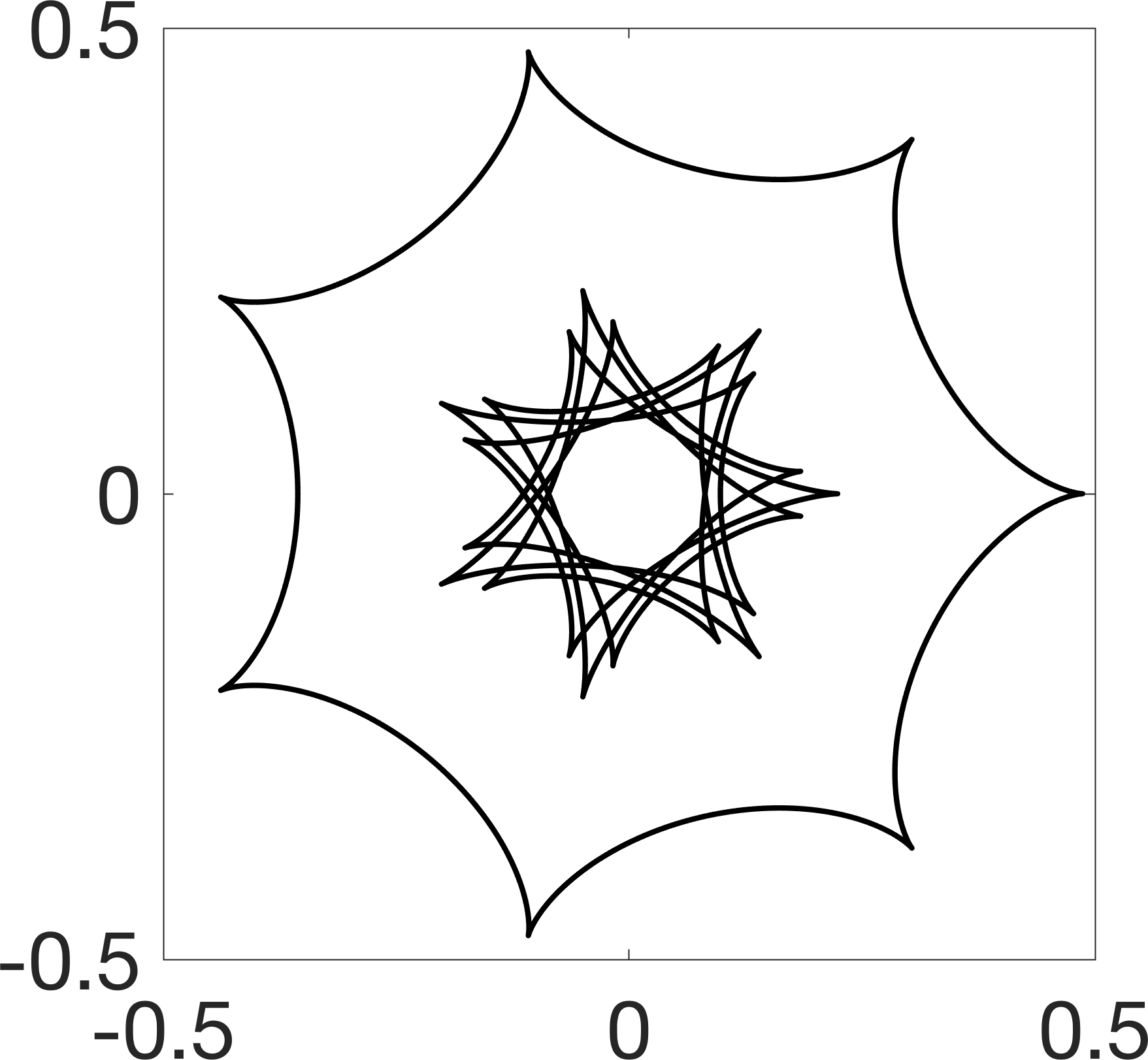}
\includegraphics[width=0.338\linewidth]{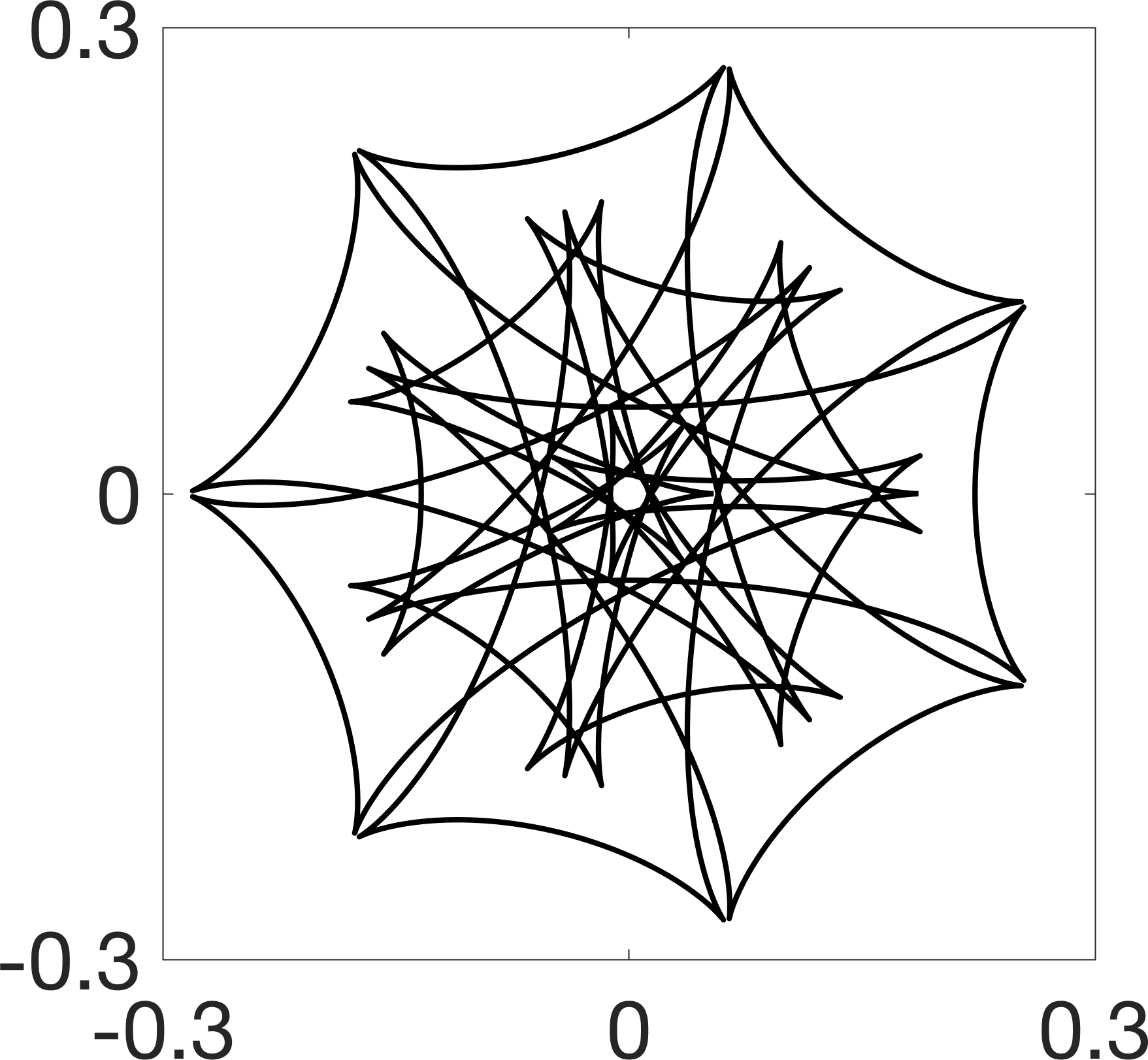}
\caption{Caustics of harmonic mappings: \eqref{eqn:wilmshurst} with $n=3$ 
(left), \eqref{eqn:mpw} with $n=7$ and $\rho = 0.7$ (middle), \eqref{eqn:rhie} 
with $n=7$, $\rho = 0.7$ and $\eps = 0.1$ (right).}
\label{fig:exp_caus}
\end{figure}

The following computations have been performed in MATLAB R2019b on an i7-7700 
\@ 4 $\times$ 3.60~GHz CPU with 16~GB~RAM using our 
implementation\footnote{\url{https://github.com/transportofimages/}}, which 
also contains m-files to reproduce all examples.

\begin{table}[t]
\centering
\begin{tabular}{|lrrrrrr|} \hline
Function & Zeros & Max. res. & Time & N.\,iter. & Steps & Ref. \\ \hline
\eqref{eqn:example_log} & 4 & 6.6613e-16 & 36 ms & 237 & 8 & 2 \\
\eqref{eqn:wilmshurst} & 9 & 1.7764e-15 & 44 ms & 291 & 8 & 1 \\
\eqref{eqn:mpw} & 22 & 8.9509e-16 & 62 ms & 1646 & 19 & 4 \\
\eqref{eqn:rhie} & 35 & 7.8505e-16 & 85 ms & 3851 & 
40 & 9 \\
\hline
\end{tabular}
\caption{The transport of images method for~\eqref{eqn:example_log}, 
\eqref{eqn:wilmshurst} with $n=3$, \eqref{eqn:mpw} with $n=7$, $\rho = 0.7$, 
and \eqref{eqn:rhie} with $n=7$, $\rho = 0.7$, $\eps = 0.1$: maximal residual 
(max.\ res.), computation time, number of harmonic Newton iterations 
(N.\,iter.), number of transport steps, and number of step refinements (Ref.).}
\label{tab:1st_exp}
\end{table}

First, we compute the zeros of~\eqref{eqn:wilmshurst}, 
\eqref{eqn:mpw}, \eqref{eqn:rhie}, and of the transcendental harmonic 
mapping~\eqref{eqn:example_log} with the transport of images method.
Here we fix the angle $\theta = \pi/50$ in the construction of the transport 
path (see Section~\ref{sect:transport_path}).  The results are displayed in 
Table~\ref{tab:1st_exp}.
We make the following three main observations:
(1) The transport of images method finishes with the correct number 
of zeros for each function, in particular for the transcendental harmonic 
mapping~\eqref{eqn:example_log}.
The logarithmic term makes a symbolic computation of the zeros difficult,
e.g., Mathematica cannot determine the zeros of this function.
We emphasize that the number of zeros is not an input parameter of our method.
(2) The results are computed within a fraction of a second on a standard 
computer.
(3) In each case the computed results are very accurate, as shown by the 
residual $\abs{f(z_j)}$ at the computed zeros, which is in the order of the 
machine precision.  Table~\ref{tab:1st_exp} shows the maximal residual over all 
computed zeros for the respective functions.
Furthermore, we see that a step refinement as described in 
Section~\ref{sect:transport_path} was necessary for all functions.
The number of (harmonic) Newton iterations also includes iterations that are 
necessary to decide whether a step has to be refined.

\begin{figure}[t!]
\centering{\includegraphics[width=\linewidth]{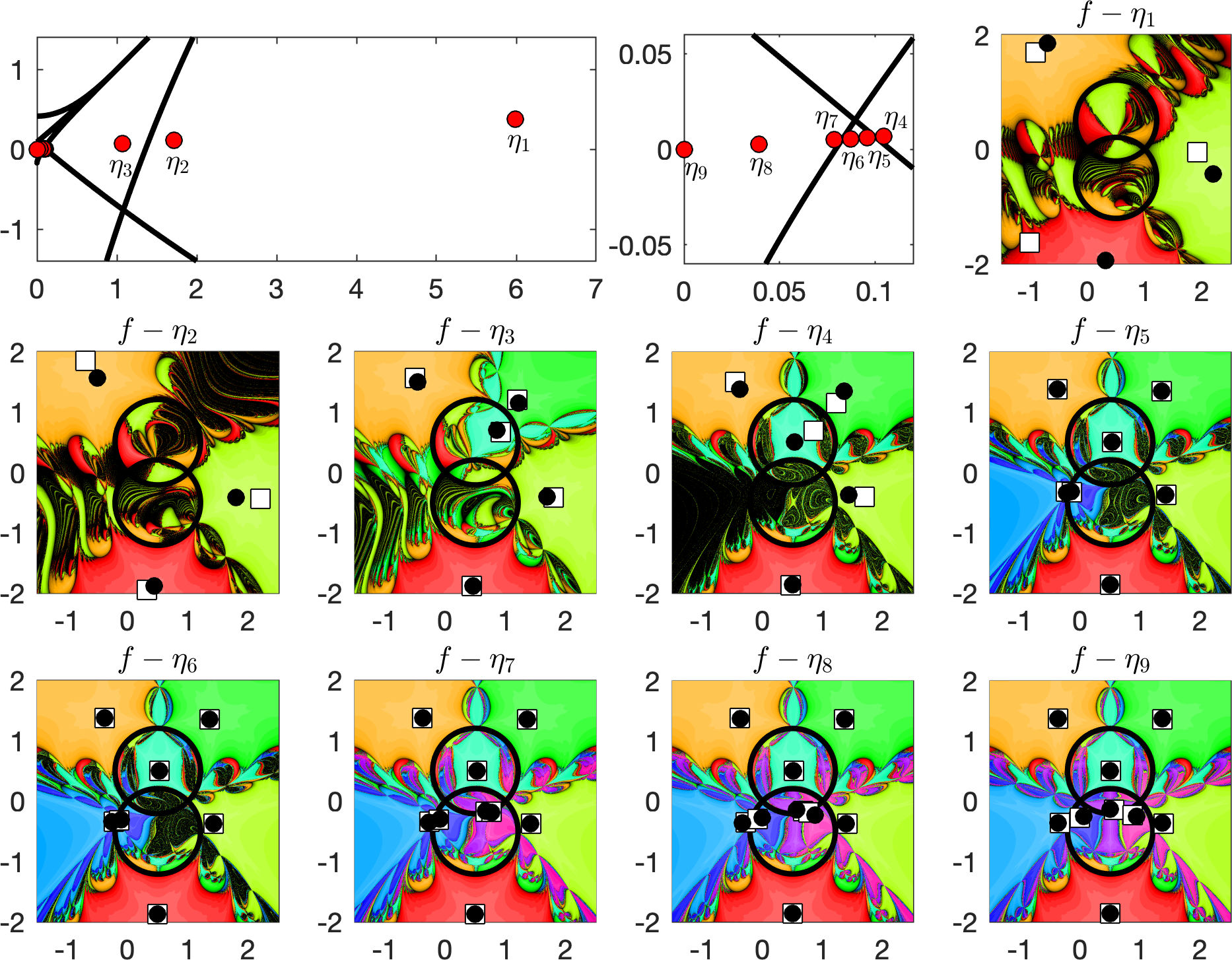}}
\caption{Transport of images method for~\eqref{eqn:wilmshurst} with $n=3$: 
Transport path and dynamics of $H_{f-\eta_k}$ with prediction set 
(white squares) and zeros of $f - \eta_k$ (black dots).}
\label{fig:wilmshurst_basins}
\end{figure}

Figure~\ref{fig:wilmshurst_basins} illustrates the transport of images method 
for~\eqref{eqn:wilmshurst} with $n=3$.
Panels~1 and~2 display the transport path and a zoom-in close to the origin.
The remaining panels show the 
dynamics of the harmonic Newton maps $H_{f-\eta_k}$ for $k = 1, 2, \ldots, 9$.
Recall from Remark~\ref{rem:basins} that points which are attracted by the same 
solution have the same color.
The initial phase is visualized in the top right plot.
Every solution of $f(z) = \eta_1$ (black dots) attracts exactly one of the 
initial points (white squares), constructed as in 
Section~\ref{sect:initial_phase}, hence these initial points form a prediction 
set of $f^{-1}(\{ \eta_1 \})$.
For aesthetic reasons we have set $\eta_1 = 6e^{i \pi/50}$ by hand.
The remaining panels visualize the steps in the transport phase, again with the 
solutions of $f(z) = \eta_k$ (black dots) and the prediction sets (white 
squares).
In particular, we see how the basins of attraction evolve while $\eta_k$ 
`travels' along the transport path.
New basins and their respective solutions appear in pairs for $\eta_3$ (green 
and cyan), $\eta_5$ (blue and light blue) and $\eta_7$ (pink and purple), right 
after a caustic crossing, as predicted by the theory.
For each step the prediction set of $f^{-1}(\{ \eta_{k+1} \})$ consists of 
$f^{-1}(\{ \eta_k \})$ (black dots in the previous panel), and of the points 
$z_\pm$ from Theorem~\ref{thm:zero_at_crit} in case of a caustic crossing.
We see that the `new' solutions are close to the predicted points $z_\pm$.
Note that Theorem~\ref{thm:correction_regular} does not apply to the step from
$\eta_7$ to $\eta_9$.  Therefore, this step was divided by introducing the point
$\eta_8 = (\eta_7 + \eta_9)/2$.
A transport path with more step refinements, but for the 
function~\eqref{eqn:rhie}, is displayed in Figure~\ref{fig:exp_transportpath}.

\begin{figure}[t!]
\centering{\includegraphics[width=\linewidth]{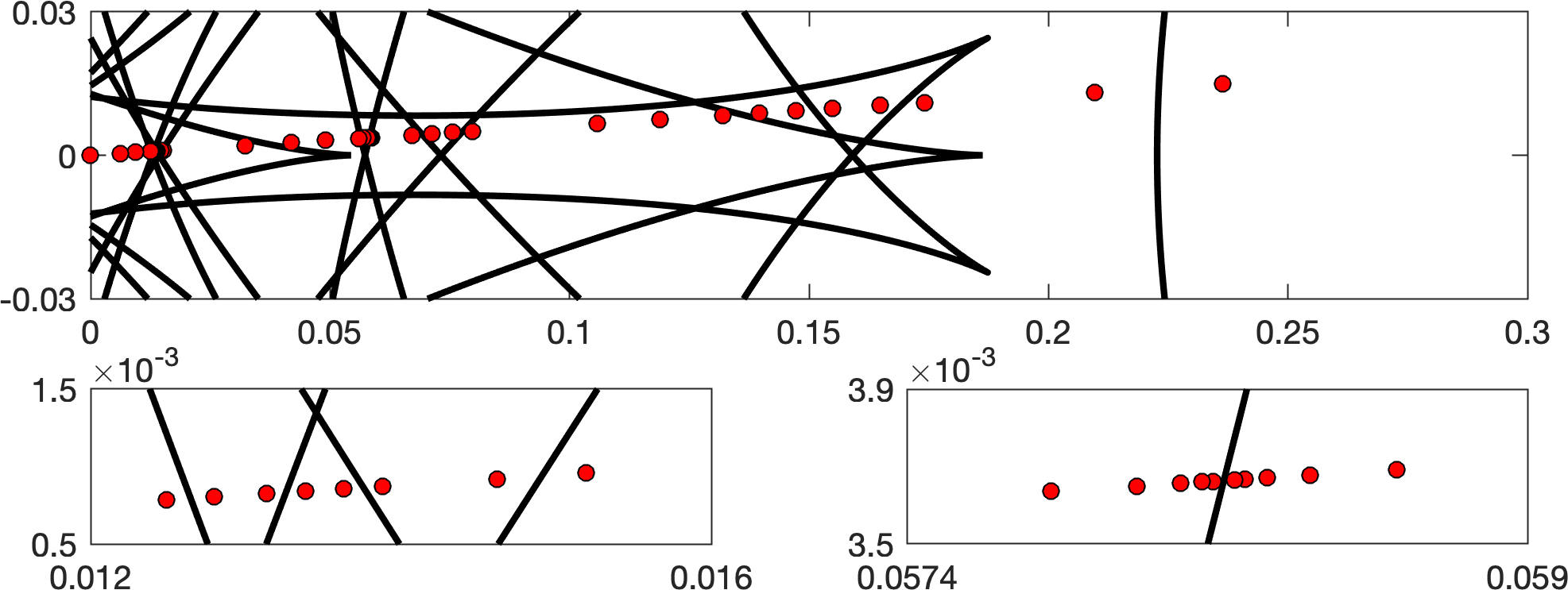}}
\caption{Transport path for~\eqref{eqn:rhie} with $n=7$, $\rho= 0.7$, $\eps = 
0.1$ and $\theta = \frac{\pi}{50}$.} \label{fig:exp_transportpath}
\end{figure}

Next, we compute the zeros of the function~\eqref{eqn:mpw} for different values 
of~$n$ with the transport of images method (without fixing the angle $\theta$).
For each $n = 3, 4, \ldots, 12$ we consider $50$ instances of~\eqref{eqn:mpw} 
with randomized $\rho \in \co{0.7, \rho_c}$, 
and apply our method to these functions.
Similarly, we consider $50$ instances of the function~\eqref{eqn:rhie} with 
random $\rho \in \co{0.7, \rho_c}$ and $\eps \in \cc{\frac{\eps^*}{2}, 
\frac{3\eps^*}{4}}$.
The quantities $\rho$ and $\eps$ are uniformly distributed in the respective 
intervals.
Table~\ref{tab:random_instances} shows the computed number of zeros, the number 
of harmonic Newton iterations, the number of restarts with a new angle 
$\theta$ and the computation time.
For each $n$, the quantities are averaged over all instances.
We observe the following.
First of all, the transport of images method terminates with the correct number 
of zeros for all $n$ and all instances.
For most instances the transport along the first ray is already 
successful.  For the remaining ones a restart with a new angle is necessary to 
compute the zeros.
In this example, computing the caustics takes roughly half of the 
overall computation time.
This can be exploited when solving $f(z) = \eta_k$ for several $\eta_k \in 
\C$, which is a scenario in the astrophysical application 
in~\cite[Sect.~10.5]{SchneiderEhlersFalco1999}.

\begin{table}[t!]
\centering
\mbox{
\hspace{-.36cm}
\begin{tabular}{|rrrrr|}
\hline
$n$ & Zer. & N.\,iter. & Rest. & Time \\ \hline
3 & 10 & 543.98 & 0.00 & 37 ms \\
4 & 13 & 926.68 & 0.00 & 45 ms \\
5 & 16 & 1328.06 & 0.02 & 50 ms \\
6 & 19 & 1432.50 & 0.00 & 53 ms \\
7 & 22 & 1765.74 & 0.02 & 60 ms \\
8 & 25 & 2253.86 & 0.02 & 68 ms \\
9 & 28 & 2653.40 & 0.00 & 73 ms \\
10 & 31 & 3126.28 & 0.04 & 81 ms \\
11 & 34 & 3393.04 & 0.00 & 88 ms \\
12 & 37 & 3926.20 & 0.02 & 97 ms \\
\hline
\end{tabular}
\hfill
\centering
\begin{tabular}{|rrrrr|}
\hline
$n$ & Zer. & N.\,iter. & Rest. & Time \\ \hline
3 & 15 & 948.46 & 0.00 & 44 ms \\
4 & 20 & 1613.32 & 0.02 & 55 ms \\
5 & 25 & 2046.24 & 0.02 & 61 ms \\
6 & 30 & 3014.40 & 0.00 & 70 ms \\
7 & 35 & 3590.90 & 0.04 & 81 ms \\
8 & 40 & 4989.04 & 0.00 & 93 ms \\
9 & 45 & 5820.60 & 0.06 & 107 ms \\
10 & 50 & 6941.06 & 0.06 & 120 ms \\
11 & 55 & 8035.50 & 0.10 & 137 ms \\
12 & 60 & 9154.42 & 0.04 & 147 ms \\
\hline
\end{tabular}}
\caption{The transport of images method averaged over $50$ instances.
Left:~\eqref{eqn:mpw} for $\rho \in \co{0.7, \rho_c}$. 
Right:~\eqref{eqn:rhie} for $\rho \in \co{0.7, \rho_c}$ 
and $\eps \in \cc{\frac{\eps^*}{2},\frac{3\eps^*}{4}}$.
Number of harmonic Newton iterations (N.\,iter.), number of restarts with new 
angle $\theta$ (Rest.), computation time.}
\label{tab:random_instances}
\end{table}

\begin{table}[t!]
\centering
\begin{tabular}{|lcccc|} \hline
Function & Zeros & Max. res. & Restarts & Time \\ \hline
\eqref{eqn:mpw} & 301 & 6.4737e-16 & 0 & 5.36 s \\
\eqref{eqn:rhie}& 125 & 1.3552e-15 & 7 & 
2.52 s \\
\hline \end{tabular}
\caption{The transport of images method for \eqref{eqn:mpw} with $n=100$, $\rho = 0.94$, and \eqref{eqn:rhie} with $n=25$, $\rho = 0.9$, $\eps = 0.4$.}
\label{tab:large_n}
\end{table}

\begin{figure}
{\centering
\includegraphics[width=0.327\linewidth]{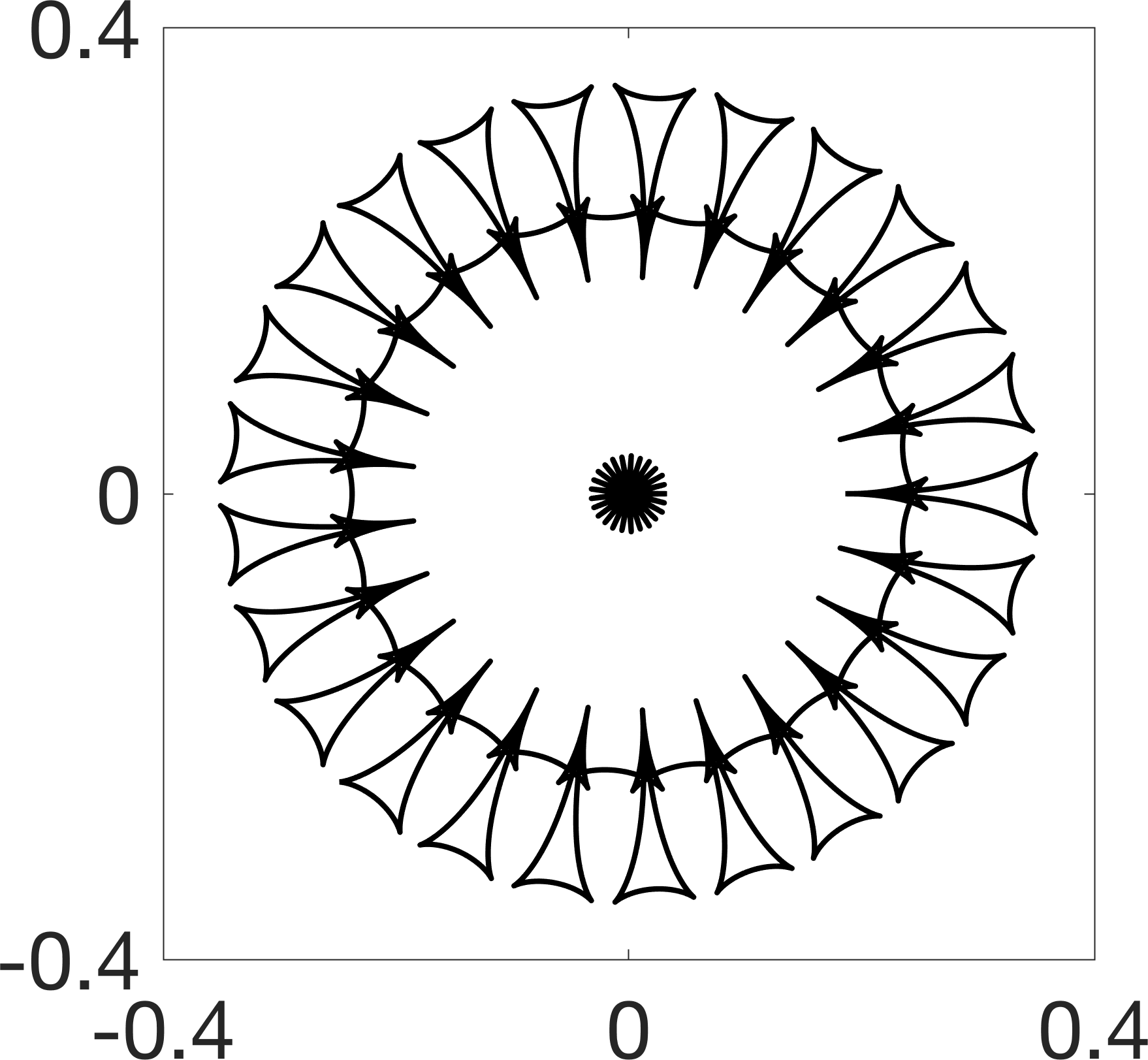}
\includegraphics[width=0.342\linewidth]{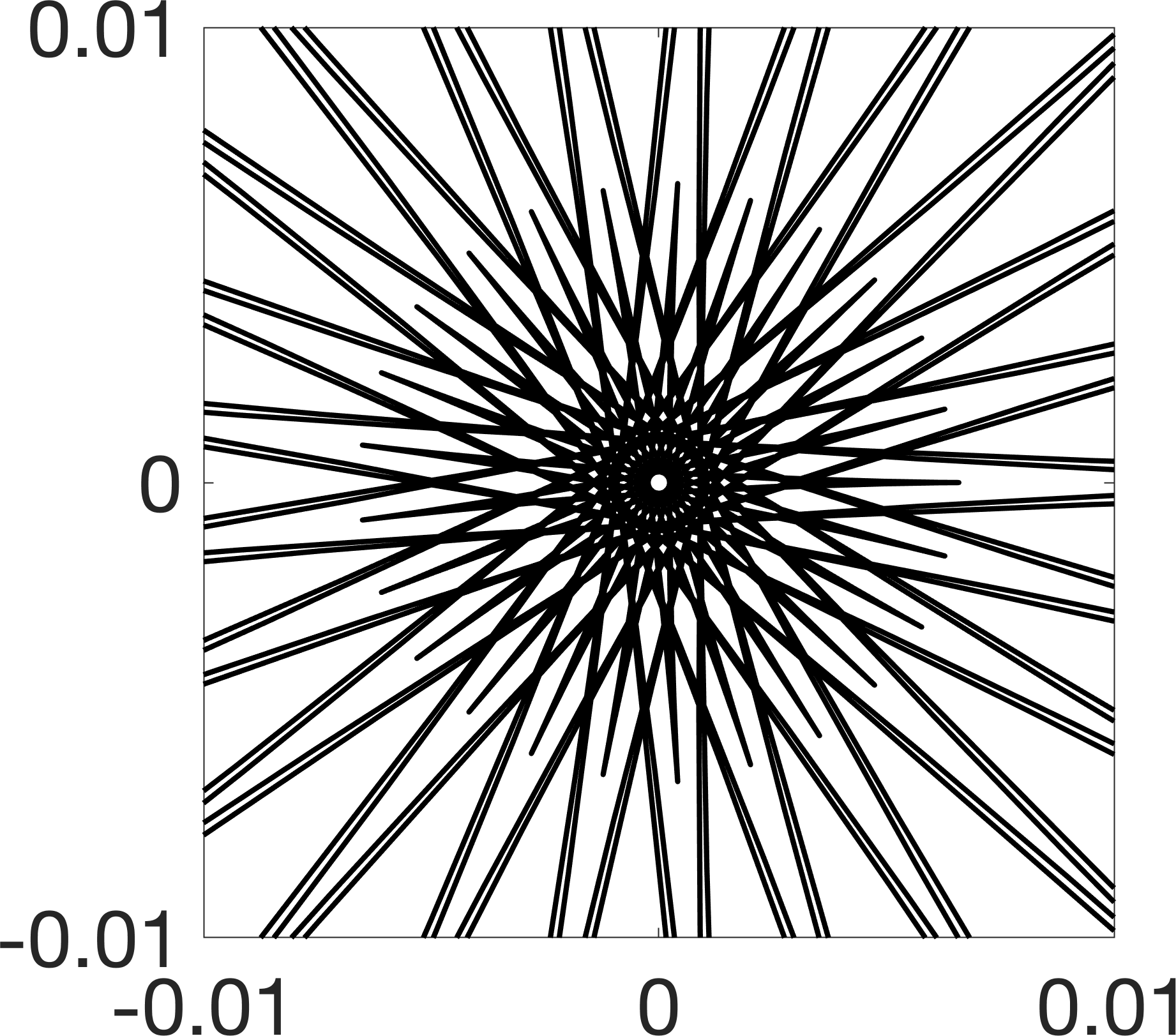}
\includegraphics[width=0.295\linewidth]{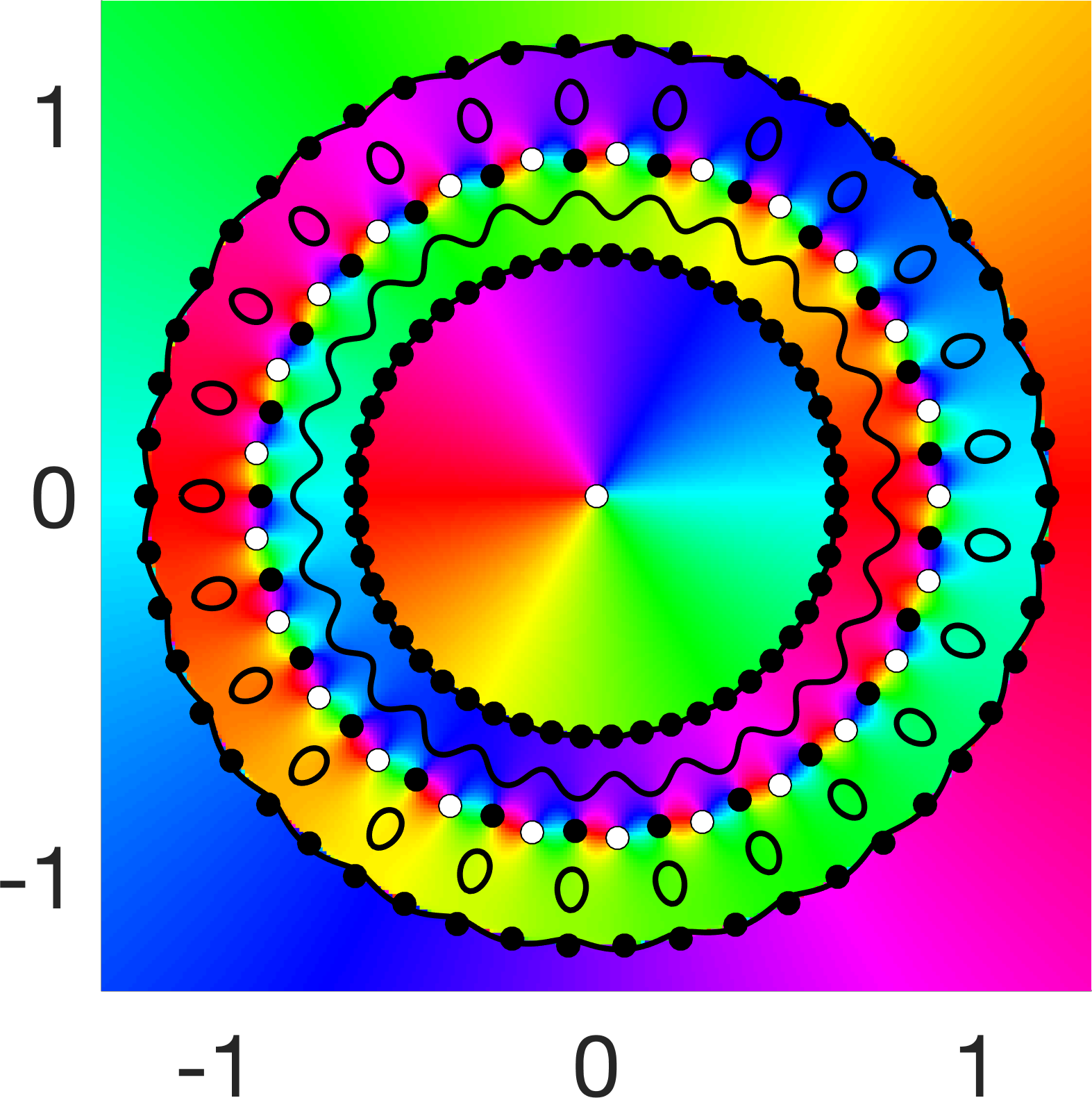}
}
\caption{
The harmonic mapping $f(z) = z - \conj{\left(0.6\frac{z^{24}}{z^{25} - 
0.9^{25}} + 0.4\frac{1}{z}\right)}$; see~\eqref{eqn:rhie}. Left: caustics.  
Middle: zoom 
of caustics. Right: phase 
plot with critical curves, zeros (black dots) and poles (white dots).}
\label{fig:large_n}
\end{figure}

In our next example we consider~\eqref{eqn:mpw} with $n = 100$ and $\rho = 
0.94$, and~\eqref{eqn:rhie} with $n = 25$, $\rho = 0.9$ and $\eps = 0.4$.
Note that every transport path along a ray $R_\theta$ intersects the caustics 
at least $100$ times for the first function and at least $50$ times for the 
second function.  The
transport of images method computes all zeros in a few seconds on a standard 
computer; see Table~\ref{tab:large_n}.
The residuals are close to the machine precision.
In contrast to the examples in Table~\ref{tab:random_instances}, seven restarts 
with a new angle $\theta$ are necessary for~\eqref{eqn:rhie} with $n = 25$.
This is due to the extremely nested caustics; see Figure~\ref{fig:large_n} 
(left and middle).
The $125$ zeros and $26$ poles of the second function are displayed in a phase 
plot in Figure~\ref{fig:large_n}.  In a phase plot of $f$ the complex plane is
colored according to the phase $f(z)/\abs{f(z)}$; see~\cite{Wegert2012} for an 
extensive discussion of phase plots and their applications.

As mentioned in the introduction we are not aware of any specialized software 
to compute \emph{all} zeros of harmonic mappings.
However, we compare our method with the general purpose root finder
from~\cite{NakatsukasaNoferiniTownsend2015}, which is based on B\'ezout 
resultants.  There are two implementations of 
this method: \texttt{rootsb}\footnote{rootsb, version of March 20, 2021, \url{ 
https://de.mathworks.com/matlabcentral/fileexchange/44084-computing-common-zeros
-of-two-bivariate-functions}.} and the \texttt{roots} command of 
Chebfun2\footnote{Chebfun2, version of September 30, 
2020, \url{www.chebfun.org}.}.
Chebfun2~\cite{TownsendTrefethen2013} is the state-of-the-art toolbox for 
numerical computation with (real or complex) smooth and bounded functions on a 
rectangle in the plane.
The idea of~\cite{NakatsukasaNoferiniTownsend2015} is to find the zeros of two 
(real-valued) smooth functions $f(x,y) = g(x,y) = 0$ by first computing 
polynomial approximants $p(x,y)$ and $q(x,y)$ that are accurate to machine 
precision in the supremum norm relatively to $f$ and $g$, respectively, and then 
solving the polynomial system $p(x,y) = q(x,y) = 0$.
In this procedure the polynomial approximants are locally resampled if 
appropriate.  While \texttt{rootsb} uses the original functions, Chebfun2 uses 
the computed interpolant for the resampling.

\begin{figure}[t]
\includegraphics[width=0.485\linewidth]{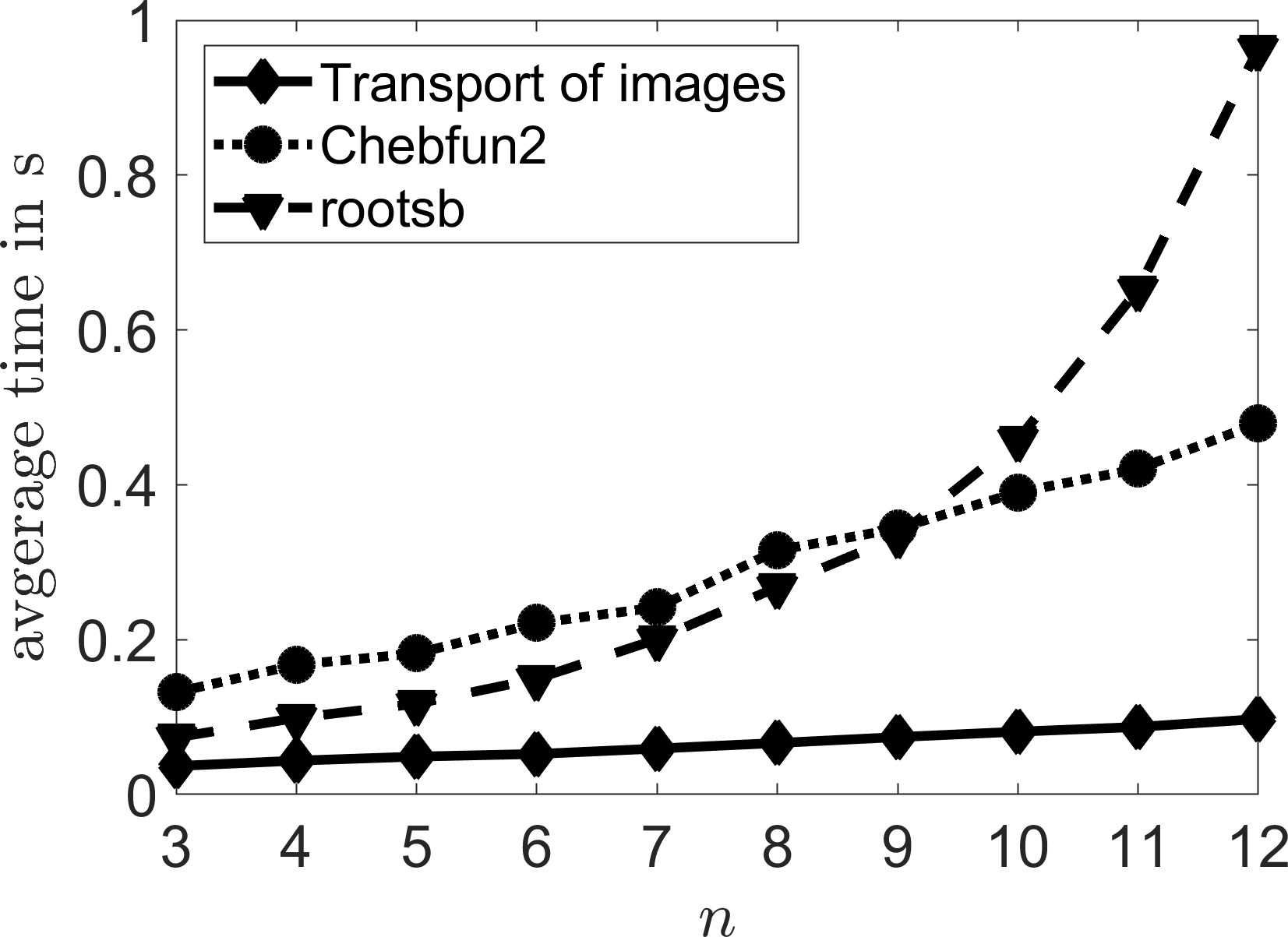}
\hfill
\includegraphics[width=0.485\linewidth]{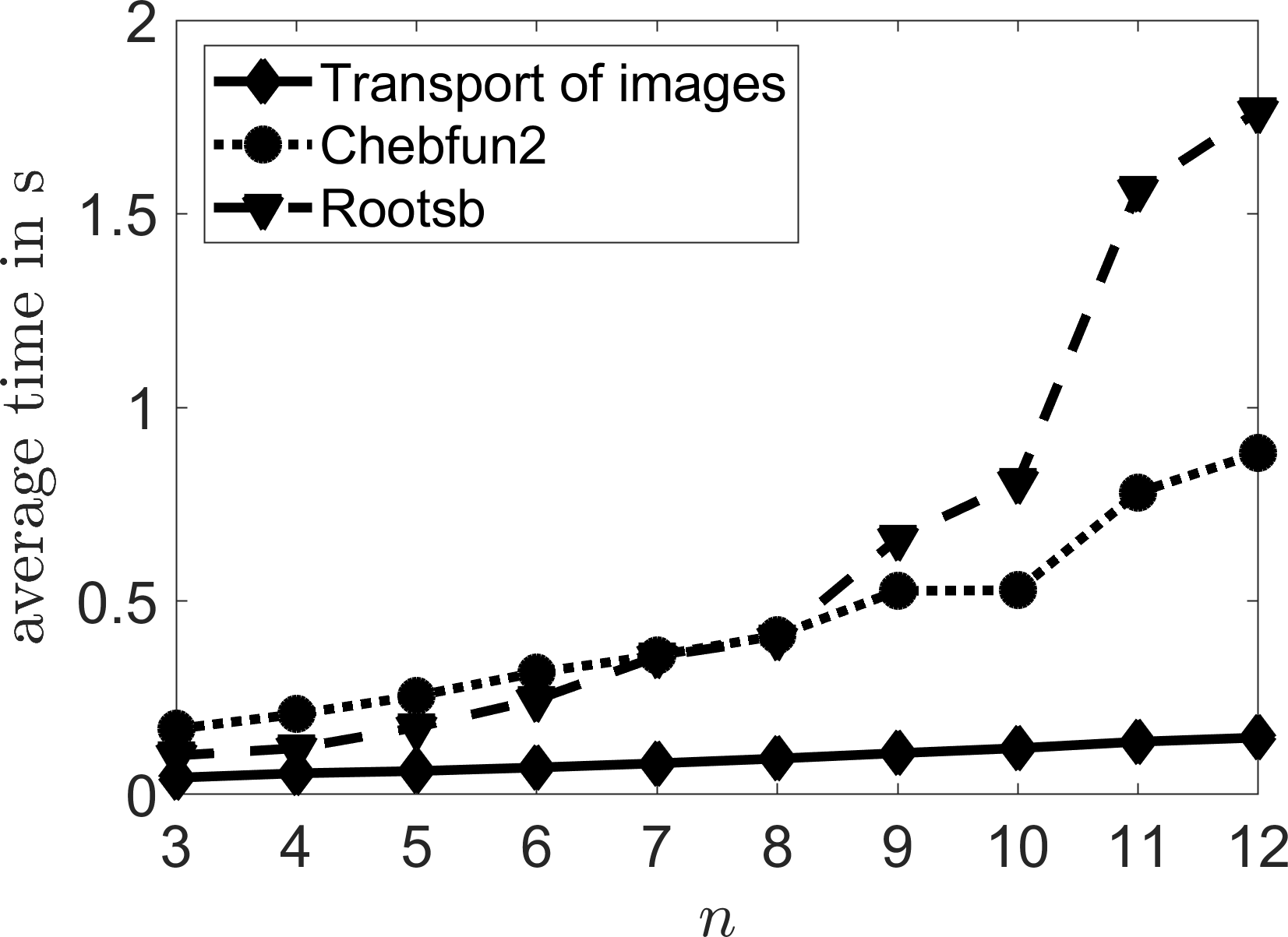}
\caption{Timings of the transport of images method, Chebfun2 and rootsb averaged over the $50$ instances for~\eqref{eqn:mpw} (left) 
and~\eqref{eqn:rhie} (right) as in Table~\ref{tab:random_instances}.}
\label{fig:chebfun}
\end{figure}

\begin{table}[t!]
\centering
\begin{tabular}{|l|rr|rr|rr|} \hline
Function &  \multicolumn{2}{c|}{Transport of images} & 
\multicolumn{2}{c|}{Chebfun2} & \multicolumn{2}{c|}{rootsb} \\ 
 &  Zeros  & Time &  Zeros & Time &  Zeros & Time  \\ \hline
\eqref{eqn:mpw} & 301  & 5.36 s & 10017 &  45.73 s & 301  & 60.93 s\\
\eqref{eqn:rhie} & 125  & 2.52 s & 95 &  9.98 s & 75  & 10.78 s   \\
\hline  \end{tabular}
\caption{Comparison of the transport of images method with Chebfun2 and 
\texttt{rootsb} for the functions \eqref{eqn:mpw} with $n=100$, $\rho = 0.94$, 
and \eqref{eqn:rhie} with $n=25$, $\rho = 0.9$, $\eps = 0.4$, see also 
Table~\ref{tab:large_n}.}
\label{tab:large_n_chebfun}
\end{table}

We consider again the functions in Table~\ref{tab:random_instances}.
To compute the zeros of~\eqref{eqn:mpw} and~\eqref{eqn:rhie} with Chebfun2 and 
\texttt{rootsb} we 
multiply them by their respective denominators to remove the poles and thus 
obtain the functions
\begin{equation*}
\begin{split}
F_1(z) &= z (\conj{z}^n - \rho^n) - \conj{z}^{n-1}, \\
F_2(z) &= (\abs{z}^2 - 1) \conj{z}^n + (\eps - \abs{z}^2) \rho^n.
\end{split}
\end{equation*}
We use the square $R = \cc{-1.2, 1.2} \times \cc{-1.2, 1.2}$, 
which contains all zeros of $F_1$ and $F_2$ for all $n = 3, 4, \ldots, 12$ and 
all $\rho$ and $\eps$ as above.
Note that the supremum norm of $F_j$ on $R$ grows exponentially with $n$.
Despite this difficulty, Chebfun2 and \texttt{rootsb} compute all zeros of 
$F_1$ for all instances.
For $F_2$, \texttt{rootsb} computes all zeros of all instances, and Chebfun2 
returns the correct number of zeros for all instances for 
$n = 3, 4, \ldots, 9$ and $n = 11$, but a wrong 
number of zeros for $5$ out of $50$ instances for $n = 10$ or $n = 12$.
This may be due to the fact that the magnitude of $F_2$ (and also $F_1$) varies 
highly, which is known as the `dynamical range issue'.  Here, resampling 
with the original function in \texttt{rootsb} leads to more accurate results 
at the expense of speed; see Figure~\ref{fig:chebfun}.
However, for larger $n$ as in Table~\ref{tab:large_n_chebfun}, 
\texttt{rootsb} also does not compute the correct zeros.
This highlights the ill-behavedness of the considered functions
and how difficult it is to compute their zeros.
The problem-adapted transport of images method finishes in all cases with 
the correct number of zeros; see Table~\ref{tab:random_instances}.
Moreover, it is much faster than Chebfun2 and \texttt{rootsb} in these 
examples; see Figure~\ref{fig:chebfun} and Table~\ref{tab:large_n_chebfun}.

\section{Summary and outlook}\label{sect:summary}

We established and analyzed the transport of images method, which is the first
problem-adapted method to compute \emph{all} zeros of a (non-degenerate)
harmonic mapping $f$. Our method is guaranteed to find all zeros of $f$, even 
though this number is a priori unknown,
provided that no zero of $f$ is singular.
Moreover, our MATLAB implementation
performs remarkably well for several functions from the literature.
While we focused on the computation of zeros in our examples, we can also 
compute all solutions of $f(z) = \eta$ for any point $\eta \in \C \setminus 
f(\cC)$.
For this, we either consider $f - \eta$, or replace the transport path 
from $\eta_1$ to $0$ by a transport path from $\eta_1$ to $\eta$ (with 
$\eta_1$ from the initial phase).
More generally, we can compute all solutions of $f(z) = \eta$ given all 
solutions of $f(z) = \xi$ (with any points $\eta, \xi \in \C \setminus f(\cC)$) 
and the 
caustics of $f$, by transporting the solutions along a transport path from 
$\xi$ to~$\eta$.
This allows us to deduce the solutions for all right-hand sides 
from the local behavior of $f$ on the critical curves.

Further studies of the transport of images method could include different 
constructions of transport paths (e.g., step size control), improvements in the 
computation of the caustics, as well as the 
investigation of other algorithms as corrector.  Generalizing the transport of 
images method to a broader class of functions, e.g., for harmonic mappings on 
domains with boundary (see, e.g.,~\cite{BergweilerEremenko2010}), 
might also be of interest.
Finally, adjusting the transport of images method to the special 
functions of gravitational lensing should be of particular interest.

\paragraph*{Acknowledgements}
We thank J\"org Liesen for several helpful suggestions on the manuscript.
Moreover, we are grateful to the anonymous referees for valuable comments, 
which lead to improvements of this work.

\footnotesize
\bibliography{transport_of_images}
\bibliographystyle{siam} 

\end{document}